\documentclass[12pt,a4paper,reqno]{amsart}
\usepackage{amsmath}
\usepackage{amsfonts}
\usepackage{amssymb}
\usepackage{hyperref}
\usepackage{amsthm}
\usepackage{tikz}
\usepackage{enumerate}
\numberwithin{equation}{section}

\theoremstyle{plain}

\newcommand{\subs}{\subseteq}
\newcommand{\setm}{\setminus}
\newcommand{\lf}{\lfloor}
\newcommand{\rf}{\rfloor}
\newcommand{\lc}{\lceil}
\newcommand{\rc}{\rceil}
\newtheorem{theorem}{Theorem} [section]
\newtheorem{lemma}[theorem]{Lemma}
\newtheorem{obs}[theorem]{Observation}
\newtheorem{prop}[theorem]{Proposition}

\newtheorem{fact}[theorem]{Fact}
\newtheorem{claim}[theorem]{Claim}
\newtheorem{cor}[theorem]{Corollary}
\newtheorem{conj}[theorem]{Conjecture}
\newtheorem*{conj*}{The separation conjecture}
\theoremstyle{ition}
\newtheorem{defn}[theorem]{ition}
\newcommand{\fu}{f^+}
\newcommand{\fl}{f^-}
\newcommand{\N}{\mathbb N}
\newcommand{\cM}{\mathcal M}
\newcommand{\B}{\mathbb B}

\newcommand{\F}{\mathcal F}
\newcommand{\St}{\operatorname{St}}

\begin{document}

\title{On the separation conjecture in Avoider--Enforcer games}

\author{Ma\l gorzata Bednarska-Bzd\c ega}
\address{Faculty of Mathematics and CS, Adam Mickiewicz University in Pozna\'n}
\email{mbed@amu.edu.pl}

\author{Omri Ben-Eliezer}
\address{School of Computer Science, Tel-Aviv University}
\email{omrib@mail.tau.ac.il}

\author{Lior Gishboliner}
\address{School of Mathematics, Tel-Aviv University}
\email{liorgis1@post.tau.ac.il}

\author{Tuan Tran}
\address{Department of Mathematics, ETH, 8092 Zurich}
\email{manh.tran@math.ethz.ch} 

\maketitle

\begin{abstract}
Given a fixed graph $H$ with at least two edges and positive integers $n$ and $b$, the strict $(1 \colon b)$ Avoider--Enforcer $H$-game, played on the edge set of $K_n$, has the following rules: In each turn Avoider picks exactly one edge, and then Enforcer picks exactly $b$ edges. Avoider wins if and only if the subgraph containing her/his edges is $H$-free after all edges of $K_n$ are taken. 

The \emph{lower threshold} of a graph $H$ with respect to $n$ is the largest $b_0$ for which Enforcer has a winning strategy for the $(1\colon b)$ $H$-game played on $K_n$ for any $b \leq b_0$, and the \emph{upper threshold} is the largest $b$ for which Enforcer wins the $(1 \colon b)$ game.
The separation conjecture of Hefetz, Krivelevich, Stojakovi\'c and Szab\'o states that for any connected $H$, the lower threshold and the upper threshold of the Avoider--Enforcer $H$-game played on $K_n$ are not of the same order in $n$.  
Until now, the conjecture has been verified only for stars, by Grzesik, Mikala\v{c}ki, Nagy, Naor, Patkos and Skerman.

We show that the conjecture holds for every connected graph $H$ with at most one cycle (and at least two edges), with a polynomial separation between the lower and upper thresholds. 
We also prove an upper bound for the lower threshold of \emph{any} graph $H$ with at least two edges, and show that this bound is tight for all graphs in which each connected component contains at most one cycle. Along the way, we establish number-theoretic tools that might be useful for other problems of this type.
\end{abstract}

\section{Introduction} \label{intro}

\emph{Positional games} are two-player combinatorial games played on a board that is usually finite, where the players alternately claim previously unclaimed elements from the board until one of the players achieves her/his winning criteria, or until all elements on the board are taken. The game is partitioned into \emph{turns}, where in each turn the first player makes her move, and then the second player makes his move. 
Examples range from recreational games, such as Tic-Tac-Toe and Hex, to games played on hypergraphs, where the elements claimed are vertices or edges. The investigation of positional games has been very active lately, as is largely covered by the 2008 book of Beck \cite{ttt} and the newer book of Hefetz, Krivelevich, Stojakovi\'c and Szab\'o \cite{HefetzKrivStoSz2014-book}.
Here we focus on a subclass of positional games, called {\it Avoider--Enforcer} games. In the following
subsections we review what is already known, before presenting our new results.

\subsection{To win or not to win}
Perhaps the most widely investigated type of games in the graph case is \emph{Maker--Breaker} games. Here, Maker wins if by the end of the game, the graph consisting of her edges contains a desirable \emph{winning set} of edges (e.g. the edges of a triangle), and Breaker wins otherwise. For example, in the Maker--Breaker $H$-game (for a predetermined graph $H$), Maker wins if and only if, after all edges have been taken, her graph contains an $H$-copy. 
Naturally, it is interesting to study the {\em mis\`{e}re} version of Maker--Breaker games. The mis\`{e}re version of a game is played according to the original rules, but the goal of the game is to {\em lose}.
Mis\`{e}re Maker--Breaker games are called \emph{Avoider--Enforcer} games: here, Avoider \emph{loses} if by the end of the game, the graph consisting of her edges contains an undesirable \emph{losing set} of edges, 
and wins otherwise. For example, in the Avoider--Enforcer $H$-game, Avoider loses if her graph in the end of the game contains an $H$-copy.
Our focus in this paper is on Avoider--Enforcer games, see e.g.~\cite{acyclic, planar,ae10,ae07,cover} for some previous results on this type of games.

It turns out that if each player picks exactly one edge per turn then Enforcer wins the $H$-game easily for every $H$ (provided that the board is large enough). 
Thus, it is natural to consider a \emph{biased} variant, where Avoider only picks one edge per turn, and Enforcer picks $b$ edges per turn. 
The general question here is to understand, for an Avoider--Enforcer $H$-game with given parameters $n$ and $b$, whether Avoider has a winning strategy for this game, or Enforcer has such a strategy. Note that in games of this type, where the game board is finite and the outcome of the game is always a win for one of the players, that is, exactly one of the players has a winning strategy. This is described in more detail in \cite[Section 1.2]{HefetzKrivStoSz2014-book}. 

The biased $H$-game, and in particular the triangle game, i.e. the case where $H = K_3$, has attracted a considerable amount of interest, both in the Maker--Breaker setting (here Maker picks one edge and Breaker picks $b$ edges every turn) and in the Avoider--Enforcer setting. The asymptotic behaviour of the threshold bias in the Maker--Breaker triangle game, i.e., the maximal bias $b$ for which Maker wins the triangle game on $K_n$, has been settled up to a multiplicative constant in the classical paper of Chv\'{a}tal and Erd\H{o}s \cite{ChvatalErdos1978} from 1978 (see also 
\cite[Theorem 3.1.3]{HefetzKrivStoSz2014-book}). The constant has been slightly improved by Balogh and Samotij \cite{BaloghSamotij2011}. A result of Bednarska and {\L}uczak \cite{BLuczak2000} extends \cite{ChvatalErdos1978}, determining the asymptotic behaviour of the threshold for the Maker--Breaker $H$-game for any graph $H$.
On the other hand, the Avoider--Enforcer $H$-game is much less understood, and even seemingly simple cases such as the triangle game have been wide open.

\subsection{Avoider--Enforcer $H$-game}
\label{sec:AEHgame}
Formally, we study a class of strict Avoider--Enforcer games (cf.~\cite{ttt,ae07}).
Let $H$ be a graph with at least two edges and let $n,b\in \N$. In the $(1 : b)$ \emph{strict Avoider--Enforcer $H$-game}   
played on the complete graph $K_n$, in each round the players claim previously unclaimed edges of $K_n$.
Avoider selects exactly one edge per turn and Enforcer selects exactly $b$ edges per turn. If the number of unclaimed edges is strictly less than $b$ before a move of Enforcer, then Enforcer must claim all of those edges. Avoider loses if by the end of the game she selects all edges of a copy of $H$, otherwise she wins.
Throughout the paper, for simplicity we say ``Avoider--Enforcer games'' instead of ``strict Avoider--Enforcer games",
since we do not consider the monotone version of Avoider--Enforcer games, introduced in \cite{ae10}.

Following \cite{ae10, ae07}, we consider two types of \emph{thresholds}. Given a graph $H$ with at least two edges,\footnote{If $H$ contains at most one edge, Enforcer trivially always wins, independently of the value of $b$.} {\sl the lower threshold bias} $\fl_{H}(n)$ is the largest integer such that 
for every $b\le \fl_{H}(n)$, Enforcer has a winning strategy for the $(1 : b)$  Avoider--Enforcer $H$-game on $K_n$. 
The {\sl upper threshold bias} $\fu_{H}(n)$ is the largest integer $b$ such that Enforcer wins 
the $(1 : b)$  Avoider--Enforcer $H$-game on $K_n$.
Throughout the paper the upper (lower) threshold bias is simply called the upper threshold (lower threshold, respectively). 

Our main goal in this paper is to investigate the asymptotic behaviour of $\fl_{H}(n)$ and $\fu_{H}(n)$. We always view $H$ as fixed and $n$ as a large integer (tending to infinity). 
The current state of knowledge regarding the asymptotic behaviour of $\fl_{H}(n)$ and $\fu_{H}(n)$ leaves much to be desired (cf. \cite[Section 4.6]{HefetzKrivStoSz2014-book}).
Trivially, $\fl_{H}(n) \leq \fu_{H}(n)$ always holds. 
Hefetz et al.~\cite{ae10}  showed that for the path $P_3$ 
on 3 vertices, $\fl_{P_3}(n)=\Theta(n^{3/2})$ and $\fu_{P_3}(n)=\binom n2-2$. Grzesik et al.~\cite{aestars}
generalised this result, proving that
\begin{equation}\label{thr_stars}
\fl_{S_h}(n)=\Theta\big(n^{\frac{h}{h-1}}\big)\text{ and }
\fu_{S_h}(n)=\Theta\big(n^{\frac{h-1}{h-2}}\big)
\end{equation}
for every star $S_h$ on $h\ge 4$ vertices.

To present the next set of results, we need the following definitions. Call a graph {\em unicyclic} if it is connected and contains exactly one cycle. For a non-empty graph $H$, we define the following parameters:
$$
m(H)=\max\limits_{F\subseteq H:\,v(F)\ge 1}\frac{e(F)}{v(F)}\,,\quad
m'(H)=\max\limits_{F\subseteq H:\,e(F)\ge 1}\frac{e(F)-1}{v(F)}\,,
$$
where $v(F)$, $e(F)$ are the number of vertices and edges of $F$, respectively. It is easy to see that for connected $H$, $m(H) < 1$ holds if and only if $H$ is a tree, and $m(H) = 1$ holds if and only if $H$ is unicyclic.
 
In \cite{af14} the first author showed that $\fl_{H}(n)=O(n^{1/m(H)}\ln n)$ for every graph $H$ with at least two edges. The authors of \cite{aestars} managed to remove the logarithmic factor from the previous bound when $m(H)\le 1$. In \cite{af14} the lower threshold  $\fl_{H}(n)$ was also estimated from below, but the obtained bound seemed far
from optimal. As for the upper threshold, it is known that 
\begin{equation}
\label{eqn:mprime}
\fu_{H}(n)=O(n^{\frac{1}{m'(H)}})
\end{equation} 
for every graph $H$ with at least two edges \cite{af14}. 
The authors of \cite{ae10} suspected that  
the upper and the lower thresholds are not of the same order when $H$ is connected. 
Their conjecture, which we call the \emph{separation conjecture}, is the main inspiration for our research.

\begin{conj*}[\cite{ae10}]\label{sepconj}
For every connected graph $H$ with at least two edges, one has
$$\fl_H(n)=o(\fu_H(n)).$$
\end{conj*}

In view of \eqref{thr_stars}, the conjecture is true for stars. In this paper we show that it holds for all trees with at least two edges and for all unicyclic graphs $H$ (see the remark after the statement of Theorem \ref{upperCycle}).

\subsection{Our contributions}
Our first main result is the following general upper bound on $\fl_{H}(n)$, which either extends or improves all previously known general upper bounds for $\fl_{H}(n)$.
\begin{theorem}\label{lowerH}
For every graph $H$, the lower threshold of the $H$-game satisfies
$$\fl_{H}(n)=O\big(n^{\frac{1}{m(H)}}\big).$$
\end{theorem}

We learn from \eqref{thr_stars} that the upper bound in Theorem \ref{lowerH} is tight for stars.
Our next result generalises this to a much larger class of graphs. 
\begin{theorem}\label{lowerm1}
If $H$ is a graph with at least two edges and  $m(H)\le 1$, then  $\fl_H(n)=\Omega(n^{\frac{1}{m(H)}})$.
\end{theorem}

Theorems \ref{lowerH} and \ref{lowerm1} together imply $\fl_H(n)=\Theta(n^{\frac{1}{m(H)}})$ for every graph $H$ with at least two edges and $m(H)\le 1$. Amongst other things, the proof of Theorem~\ref{lowerm1} uses a supersaturation-like result (see Lemma \ref{lem:threat_count}), which might be of independent interest.

The heart of the proof of Theorem \ref{lowerm1} for trees is the following result, which states that under some mild conditions on the bias $b$, Enforcer can force Avoider to make many threats. Below, an edge $e\in E(K_n)$ is called an {\it $H$-threat} if $e$ has not been taken by the players, and adding $e$ to Avoider's graph would create a new copy of $H$ in her graph.

\begin{theorem}\label{upper_strategy}
	Let $H$ be a tree or a unicyclic graph, with at least two edges. Then there exists a constant $\gamma=\gamma(H)\in (0,1)$ such that   
	in the $(1:b)$ Avoider--Enforcer $H$-game played on $K_n$ the following holds. 
	{If $8v(H)n\le b+1\le \gamma n^{ e(H)/(e(H)-1)}$, and if furthermore $b$ is even whenever $H$ is unicyclic}, then Enforcer has a strategy in which, 
	at some point, either he has already won or the number of $H$-threats is greater than 
	$\gamma^{e(H)-1}  n^{v(H)} / (b+1)^{e(H)-1}$.
	The above is also true if in the first round Avoider is allowed to select any number of edges she wishes (possibly none) while Enforcer has to select 
exactly $r$ edges for some fixed number $r\in \{0,1,\ldots,b\}$.
\end{theorem}

Interestingly, Theorem \ref{upper_strategy} can be proved by induction on %$H$
{$v(H)$}. Besides its %uses 
{use} in proving Theorem \ref{lowerm1}, Theorem \ref{upper_strategy} in conjunction with  some number-theoretic results gives the following lower bound on the upper threshold of trees and unicyclic graphs. 

\begin{theorem}\label{upperCycle}
For any connected graph $H$ with at least three edges the following holds.
\begin{enumerate}[{\rm (i)}]
\item \label{upper_tree}
If $H$ is a tree, then $\fu_{H}(n)=\Omega\big(n^{ \frac{v(H)-1}{v(H)-2}}\big)$. 
\item \label{upper_item1}
If $H$ is unicyclic, then $\fu_{H}(n)=\Omega\big(n^{ \frac{v(H)+2}{v(H)+1}}\big)$. 
\item \label{upper_item2}
If $H$ is unicyclic, 
then $\fu_{H}(n) \geq c_H n^{\frac{v(H)}{v(H)-1}}$ for infinitely many values of $n$, where $c_H > 0$ depends only on $H$.
\end{enumerate}
\end{theorem}

Part (\ref{upper_item2}) shows that the upper bound in \eqref{eqn:mprime} is tight for every unicyclic graph. Theorem~\ref{lowerH} together with Parts (\ref{upper_tree}) and (\ref{upper_item1}) prove the separation conjecture for all connected graphs $H$ with at least two edges and with $m(H) \leq 1$. 

\begin{cor}\label{cor:separation}
The separation conjecture holds for every graph $H$ which is either unicyclic or a tree with at least two edges. 
\end{cor}

As opposed to the lower threshold, the upper threshold for the $H$-game is affected by the number of components.
For simplicity, the statement below is given for graphs $H$ with $m(H) = 1$, but similar results hold for forests.

\begin{theorem}\label{manyUnic}
Let $H$ be a disconnected graph with $m(H)=1$ and at least two unicyclic components. Then 
\[
\fu_{H}(n)=O(n).
\]
\end{theorem}

Theorems~\ref{lowerm1} and \ref{manyUnic} imply that $f^{+}_H(n)=\Theta(f^{-}_H(n))=\Theta(n)$ for every disconnected graph $H$ with $m(H)=1$ and at least two unicyclic components. Therefore, for disconnected graphs there is no asymptotic separation between the lower and upper thresholds in general.

\subsection{Organisation}
The paper is organised as follows. In Section \ref{prelim} we introduce basic notation and simple facts that will be used throughout. Section~\ref{sec:short-proofs} is dedicated to the proofs of Theorems~\ref{lowerH}, \ref{upperCycle} and \ref{manyUnic}. We establish Theorem \ref{lowerm1} in Section \ref{sec:lowerm1}.
In Section~\ref{sec:blowup} we prove Theorem \ref{upper_strategy} by analysing games on blow-ups of multigraphs. Section~\ref{sec:conc_remarks} contains concluding remarks, including open questions and conjectures regarding the threshold biases for the $H$-game for every $H$. The appendix is devoted to the proofs of all number theoretic lemmata which are required in the proofs of the main theorems.

\section{Preliminaries}\label{prelim}
For a natural number $k$ we denote the set $\{1,2,\ldots, k\}$ by $[k]$.
For a graph $G=(V(G),E(G))$ and a set of vertices $S \subseteq V(G)$ we denote 
$E(S) = \{e \in E(G)\colon e \subseteq S\}$. For a pair of sets 
$S,T \subseteq V(G)$, let 
$E(S,T):= \big\{ \{s,t\}\in E(G)\colon s \in S, t \in T\big\}$. 
The degree of a vertex $v$ in $G$ is denoted by $d_G(v)$; $N_G(v)$ is the set of  neighbours of $v$.
We write $G[S]$ for the subgraph of $G$ induced by $S$.
We denote by 
$\Delta(G)$ and $\delta(G)$ the maximum degree and minimum degree in $G$, respectively. 

We defined Avoider--Enforcer $H$-games as games played on $K_n$. However, in subsequent
sections we will consider auxiliary Avoider--Enforcer games on other graphs. 
Therefore in all definitions below we assume that the game is played on the edge-set of a graph $F=(V(F),E(F))$. 
We call $E(F)$ the \emph{board of the game}. 
At any point in the game, the graphs $G_A$ and $G_E$ are spanning subgraphs of $F$, and their edge-sets consist
of all edges picked by Avoider and Enforcer, respectively, up to this point.  
We say that an edge $e\in E(F)$ is \emph{free} if $e \notin E(G_E) \cup E(G_A)$.
An edge $e\in E(F)$ is an \emph{$H$-threat} or simply a \emph{threat} if it is free and there exists an $H$-copy 
in $G_A \cup \{e\}$ that is not contained in $G_A$.

By default we assume that Avoider starts the game. Nonetheless, sometime we change this rule, but then
we explicitly state that Enforcer is the first player. The following two facts 
are well-known and frequently used when studying Avoider--Enforcer games. 
We will use them often in the next sections.

\begin{fact}\label{fact:threats}
Suppose that $t\ge 1$, $\binom{n}{2} - t$ is divisible by $b+1$ and the $(1:b)$ Avoider--Enforcer $H$-game 
is being played on $K_n$. If at some point of the game, the number of $H$-threats is at least $t$, then Enforcer has a strategy to win the game from this point on.
\end{fact}

\begin{fact}\label{fact:threats2}
Suppose that the $(1:b)$ Avoider--Enforcer $H$-game is being played on $K_n$. 
If at some point of the game there are at least $b+1$ $H$-threats, then Enforcer has a strategy to win the game 
from this point on. The same is true under the assumption that Enforcer is the first player.
\end{fact}

\section{Proofs of Theorems~\ref{lowerH}, \ref{upperCycle} and \ref{manyUnic}}\label{sec:short-proofs}

Here we present the proofs of Theorems~\ref{lowerH}, \ref{upperCycle} and \ref{manyUnic}. They can be read independently of each other.

\subsection{Proof of Theorem~\ref{lowerH}}
The key building blocks of the proof of Theorem~\ref{lowerH} are the following two lemmata. The first is a 
simplified version of a result due to the first author \cite{af14}.

\begin{lemma}[{implicit in \cite[Theorem 1.2]{af14}}]\label{avoiderwin}
	Let $H$ be a graph with at least one edge. Assume that $b$ and $q$ are two positive integers such that the remainder of the division of $\binom n2$ by $b+1$ is at least $q+1$, and 
	$$n^{v(H)}\cdot\left(\frac{q}{e(H)}+1\right)^{-e(H)}<1.$$
	Then Avoider has a winning strategy for the $(1 : b)$ Avoider--Enforcer $H$-game on $K_n$.
\end{lemma}

The second is a number theoretic result. We defer its proof to the appendix.

\begin{lemma}\label{bigrem}
Given real numbers $\alpha,c_1$ and $c_2$ with $\alpha > 1$ and $0<c_2<c_1$, there exists a positive constant $C=C(\alpha,c_1,c_2)$ with the following property. For any integers $q$ and $N$ with $q,N\ge C$ and $c_2 q^{\alpha}\le N \le c_1 q^{\alpha}$, 
we can find an integer $k$ such that $q<k\le Cq$ and the remainder of the division of $N$ 
by $k$ is greater than $q$.  
\end{lemma}

With Lemmata \ref{avoiderwin} and \ref{bigrem} in hand, it is easy to finish the proof of Theorem \ref{lowerH}.
\begin{proof}[\textbf{Proof of Theorem~\ref{lowerH}}.] 
	If $H$ is a matching with at least two edges, 
	then Avoider trivially wins when $b \geq \binom{n}{2}-1 = \Theta(n^2)$,
	so we can assume that $H$ contains two adjacent edges.
	Let $F$ be a subgraph of $H$ such that $e(F)/v(F)=m(H)$. Set $$q=\lceil e(F)\cdot n^{1/m(H)}\rceil=\lceil e(F)\cdot n^{v(F)/e(F)}\rceil.$$
	By the choice of $q$, we have $n^{v(F)}\left(\frac{q}{e(F)}+1\right)^{-e(F)}<1$.
	Furthermore, since $H$ contains two adjacent edges, $m(H)>1/2$. Thus we can use Lemma \ref{bigrem} with $\alpha=2m(H)$, $c_1=1$, $c_2=\frac13 e(F)^{-2m(H)}$, $q=\lceil e(F)\cdot n^{1/m(H)}\rceil$, and $N=\binom{n}{2}$ to conclude that for a suitable positive constant $C=C(H)$
there exists an integer $k=k(n)$ such that $q<k\le Cq$ and the remainder of a division of $\binom n2$ by $k$ is at least $q+1$. By Lemma \ref{avoiderwin} we conclude that 
	for $b=k-1$ Avoider has a strategy to avoid creating a copy of $F$ in $K_n$. This means that 
	she also avoids a copy of $H$, giving $\fl_{H}(n)<Cq-1=O(n^{1/m(H)})$. 
\end{proof}

\subsection{Deriving Theorem~\ref{upperCycle} from Theorem~\ref{upper_strategy}}\label{sec:upper}

In order to establish Theorem~\ref{upperCycle} via Theorem~\ref{upper_strategy}, we will need the following simple number theoretic results, whose proofs are given in the appendix.

\begin{lemma}\label{div}
Given $c>0$ and {a rational number} $\alpha\in (0,2]$, there exists $d \in (0,c)$ such that for   
infinitely many $n\in\N$ there is an odd divisor $q$ of $\binom n2-1$ with
$dn^{\alpha}\le q\le c n^{\alpha}$.
\end{lemma}

\begin{lemma}\label{div2}
Given $c>0$ and $\alpha \in (1,2)$, there exists $n_0\in \N$ such that for
each integer $n>n_0$ one can find an integer $1 \le t \le c^2 n^{2\alpha-2}$ for which $\binom n2-t$ is divisible by some odd integer $q$ with
$\frac13cn^{\alpha}\le  q\le c n^{\alpha}$.
\end{lemma}

\begin{proof}[\textbf{Proof of Theorem~\ref{upperCycle}}]
Let $h=v(H)$, and let $\gamma=\gamma(H) \in (0,1)$ be the constant given by Theorem \ref{upper_strategy}. 

    (\ref{upper_tree})
	Based on Lemma \ref{div2} with $c = \gamma$ and 
	$\alpha=\frac{h-1}{h-2}$, for $n$ large enough we {can} find  
	an integer $1 \le t \le {\gamma^2}n^{2/(h-2)}$
	and an integer $b$ so that ${n \choose 2} - t$ is divisible by $b+1$ and 
	$\tfrac13 {\gamma} n^{\frac{h-1}{h-2}} \leq b+1 \leq {\gamma} n^{\frac{h-1}{h-2}}$. 
	Note that $8h n\le b+1\le \gamma n^{(h-1)/(h-2)}$, so in view of Theorem \ref{upper_strategy}, Enforcer has a strategy such that at some point
	either he has already won or the number of threats is greater than
	$$
	\gamma^{h-2}n^{h}/(b+1)^{h-2} \ge \gamma^{h-2}n^{h}/(\gamma n^{(h-1)/(h-2)})^{h-2}={n}>t{,}
	$$
	for $h\ge 4$.
	Hereby Enforcer has a winning strategy by Fact \ref{fact:threats}.

\smallskip
(\ref{upper_item1})
Using Lemma \ref{div2} with $c = \gamma$ and 
$\alpha=\frac{h+2}{h+1}$, we infer that for $n$ large enough there exist 
an integer $1 \le t \le c^2n^{2/(h+1)}$
and an even number $b$ so that ${n \choose 2} - t$ is divisible by $b+1$ and 
$\tfrac13 c n^{\frac{h+2}{h+1}} \leq b+1 \leq c n^{\frac{h+2}{h+1}}$.  
As $8h n\le b+1\le \gamma n^{h/(h-1)}$, we learn from Theorem \ref{upper_strategy} that Enforcer has a strategy such that at some point
either he has already won or the number of threats is greater than
$$
\gamma^{h-1}n^{h}/(b+1)^{h-1} \ge \gamma^{h-1}n^{h}/(\gamma n^{(h+2)/(h+1)})^{h-1}=n^{2/(h+1)}>t.
$$
Thus Enforcer has a winning strategy by Fact \ref{fact:threats}.

\smallskip
(\ref{upper_item2})
Using Lemma \ref{div} with $c = \gamma$ and $\alpha = \frac{h}{h-1}$, we conclude that there is a constant $d=d(\gamma,h)>0$ such that for infinitely many $n \in \mathbb{N}$ there exists an even {number} $b$ for which 
$dn^{h/(h-1)} \leq b+1 \leq  \gamma n^{h/(h-1)}$ and ${n \choose 2}-1$ is divisible by $b+1$. 
By Theorem \ref{upper_strategy}, Enforcer has a strategy such that  
at some point either he has already won or there are more than $\gamma^{h-1}   n^{h} /(b+1)^{h-1}$ threats. 
In the latter case, since 
$\gamma^{h-1}  n^{h} / (b+1)^{h-1} \ge 1$,  by Fact \ref{fact:threats} Enforcer wins.
\end{proof}

\subsection{Proof of Theorem~\ref{manyUnic}}
Here we give the proof of Theorem~\ref{manyUnic}.
Our arguments rely on a result due to the first author \cite{af14}.

\begin{lemma}[{implicit in \cite[Theorem 1.2]{af14}}]\label{avoiderwin2}
	Let ${\mathcal H}$ be a finite collection of graphs with at least two vertices, and let $b$ be a natural number with
	$$
	\sum_{F\in{\mathcal H}}n^{v(F)}\cdot\Big(\frac{b}{e(F)}+1\Big)^{-e(F)}<1.
	$$
	Then Avoider has a strategy for the $(1 : b)$ Avoider--Enforcer game on $K_n$ 
	such that no copy of any graph $F\in{\mathcal H}$ appears in her graph before the last round.
\end{lemma}

\begin{proof}[\textbf{Proof of Theorem~\ref{manyUnic}}.]
	Let $H$ be a graph with at least two unicyclic components $H_1, H_2, \ldots, H_t$. Consider a $(1:b)$ Avoider--Enforcer $H$-game played on $K_n$. Let $F_1,F_2,\ldots, F_t$ be the unique cycles in $H_1, H_2, \ldots, H_t$, respectively.
	By an easy calculation one can verify that there exists a positive constant $c=c(H)$ such that
	if $b>cn$, then the assumption of Lemma~\ref{avoiderwin2} is satisfied. Thus Avoider has a strategy 
	such that for each $i\in \{1,2,\ldots,t\}$, no copy of $F_i$ appears in her graph before the last round. 
	This fact implies that at the end of the game there are no disjoint copies of $H_1$ and $H_2$ in Avoider's graph.
	Therefore Avoider wins the $H$ game on $K_n$, implying that $\fu_{H}(n)\le cn$.
\end{proof}

\section{Lower bound on the lower threshold} \label{sec:lowerm1}
This section is devoted to the proof of Theorem \ref{lowerm1}. We remark that when $H$ is a tree, our proof uses Theorem~\ref{upper_strategy} as a black box.

\subsection{Proof of Theorem \ref{lowerm1}}

We will first state special cases of Theorem \ref{lowerm1}, and then establish Theorem \ref{lowerm1} using them. For this purpose, we introduce some notation.

\begin{defn}
Every unicyclic graph $H$ contains a single cycle $v_1,\ldots,v_k$, and for every $1 \leq i \leq k$ there is a (possibly trivial)\footnote{We say that $T_i$ is trivial if $V(T_i)=\{v_i\}$.}
tree $T_i$ which contains $v_i$ and no other vertex from $v_1,\dots,v_k$, such that the trees $T_1,\ldots,T_k$ are pairwise vertex-disjoint and do not have any %edges 
{edge} between them except for the edges of the cycle $v_1,\ldots,v_k$.  
We say that an edge $e \in E(H)$ is a {\em cycle-edge} if it is one of the edges of the cycle $v_1,\ldots,v_k$. Otherwise $e$ is a {\em tree-edge}. 
The graph $H$ is called {\em odd} ({\em even}) if $k$ is odd (even). 
\end{defn}

The following results, whose proofs are given in the next three subsections, cover the case when $H$ is connected. 

	\begin{theorem}\label{lowerTree}
		If  $H$ is a tree with at least two edges, then $\fl_H(n)=\Omega(n^{1/m(H)})$. 
%  The above is true also under the assumption that Enforcer is the first player.}
	\end{theorem}

\begin{theorem}\label{thm:odd_cycle}
	Let $H$ be an odd unicyclic graph with $h$ vertices. If 
	$b + 1 \leq n/(12h)$, then Enforcer wins the $(1:b)$ $H$-game on $K_n$. 
%	The above is true also under the assumption that Enforcer is the first player.
\end{theorem}

\begin{theorem}\label{thm:even_cycle}
	Let $H$ be an even unicyclic graph with $h$ vertices, and suppose that the unique cycle in $H$ has length $k$. If 
	$b + 1 \leq n/(200k(h - k + 1))$, then Enforcer wins the $(1:b)$ $H$-game on $K_n$. 
%	The statement is true also under the assumption that Enforcer is the first player.
\end{theorem}

{We are now ready for the proof that $\fl_H(n)=\Theta(n^{1/m(H)})$ for any graph ${H}$ with $m(H)\le 1$. 	
	With the above three theorems in hand, it only remains to prove this for \emph{disconnected} graphs.}

\begin{proof}[\textbf{Proof of Theorem~\ref{lowerm1}.}]

Let $H$ be a (disconnected) graph with $m(H) \leq 1$ and $t \geq 2$ components $H_1, \ldots, H_t$ with at least one edge {\footnote{If $H$ has only one component containing edges, then forcing $H$ is equivalent to forcing this component, so the assertion follows from Theorems \ref{lowerTree}--\ref{thm:even_cycle}.}}. 
Roughly speaking, Enforcer's strategy for $H$ takes a ``divide and conquer'' approach, that partitions the vertices of $K_n$ into $t$ parts and forces Avoider to build one connected component of $H$ in each part. We provide the high-level idea of this approach, omitting some of the technical details (as these follow from the proofs of Theorems~\ref{lowerTree}, \ref{thm:odd_cycle} and \ref{thm:even_cycle}).
	
Clearly $m(H_i)\le m(H)$ for $i=1,2,\ldots,t$.
It follows from Theorems~\ref{lowerTree}, \ref{thm:odd_cycle} and \ref{thm:even_cycle}
that  there exists a constant $c>0$ such that if $b<cn^{1/m(H)}$, 
then Enforcer can force a copy of $H_i$ in the $(1 : b) $ game played on $K_{\lc n/t\rc}$, 
as well as in the game played on $K_{\lf n/t\rf}$. 
By analysing the proof of Theorems~~\ref{lowerTree}, \ref{thm:odd_cycle} and \ref{thm:even_cycle}, it is not hard to observe that Enforcer can force $H_i$ in each of these games even if he has to make a constant number, say $t^2$ in total, of extra moves, claiming at most $b$ edges in each of these moves.\footnote{Theorems~\ref{lowerTree}, \ref{thm:odd_cycle} and \ref{thm:even_cycle} cover all cases where $H_i$ is not a single edge. The case where $H_i$ is a single edge is trivial: As soon as Avoider picks one edge, Enforcer wins in this part of the game.} These extra moves affect only the constants in 
Theorems~\ref{thm:odd_cycle} and \ref{thm:even_cycle}, and in Theorem~\ref{upper_strategy} which is 
the key proof ingredient for Theorem~\ref{lowerTree}. 

Consider the following Enforcer's strategy for the $(1:b)$ $H$-game played on $K_n$. 
First Enforcer splits $V(K_n)$ into $t$ almost equal sets $V_1, V_2,\ldots, V_t$ 
and pretends he plays $t+\binom t2$ separate games $G_{i,j}$ with $i\le j$ and $i,j\in\{1,2,\ldots,t\}$.
Every game $G_{i,i}$ is played on the board $E(V_i)$, while for $i<j$ the game $G_{i,j}$ is played on $E(V_i,V_j)$.
Every time Avoider plays in $G_{i,j}$, Enforcer answers in the same game. 
In the game $G_{i,j}$ with $i<j$ Enforcer plays arbitrarily. In every game $G_{i,i}$ the aim of Enforcer is to 
force Avoider to build a copy of $H_i$. 

In each of the games $G_{i,j}$ it might happen (depending on Avoider's strategy) that Enforcer is the first player. It may also happen that in a game he has some extra (partial) moves to make but the number of such extra moves is at most the number of other games, i.e.~$t-1+\binom t2\le t^2$, and the number of extra edges Enforcer will have to add during these moves is at most $t^2 b$.
In any case, even if the above happens, Enforcer can apply the suitable strategy in each $G_{i,i}$ to force
Avoider to create a copy of $H_i$ in it. This way, Enforcer wins the $H$-game on $K_n$. 
Hence $\fl_{H}(n)\ge cn^{1/m(H)}$, as required.
\end{proof}

\subsection{Trees}
Here we derive Theorem~\ref{lowerTree} from Theorem~\ref{upper_strategy}. The proof of Theorem~\ref{upper_strategy} will be given later in Section~\ref{sec:blowup}. 
	
	\begin{proof}[\textbf{Proof of Theorem~\ref{lowerTree} from Theorem~\ref{upper_strategy}}]
		Let $H$ be a tree on $h\ge 3$ vertices and consider an $H$-game $(1 : b)$ on $K_n$. Our goal is to show that $f^-_H(n) = \Omega(n^{\frac{h}{h-1}})$. 
		It is well known that every graph with $n$ vertices and more than $hn$ edges contains a copy of every tree on $h$ vertices.
		It implies that there is a constant $c=c(h)$ such that for every $b+1\le c n$ and sufficiently large $n$,
		every Enforcer's strategy is a winning strategy. 
		
		{Consider the case $b+1 \ge 8hn$. Since we are only interested in the regime $b = O(n^{\frac{h}{h-1}})$, we may also assume that $b+1 \le \gamma n^\frac{h-1}{h-2}$} where $\gamma>0$ is the constant given by Theorem \ref{upper_strategy}.
		Then Enforcer has a strategy %such that he forces 
		{to force} either a copy of $H$ in Avoider's graph or more than
		$\gamma^{h-2}  n^{h} / (b+1)^{h-2}$ threats. As $\gamma^{h-2}  n^{h} / (b+1)^{h-2}>b$ for $8hn\le b+1\le \gamma^\frac{h-2}{h-1}n^\frac{h}{h-1}$, Enforcer wins the game for such $b$ due to Fact ~\ref{fact:threats2}. 
		
		It remains to consider the case $c n< b+1< 8hn$. Let $n'=\lfloor cn/(8h)\rfloor$.
		Then $8hn'\le b+1\le \gamma' (n')^\frac{h}{h-1}$ for every positive constant $\gamma'$, provided that $n$ is sufficiently large.
		At the first  stage of the game Enforcer selects all free edges from the set $E(K_n)\setm E(K_{n'})$, for some $K_{n'}\subs K_n$.
		Then the game is transformed into the $H$-game on $K_{n'}$, with additional assumption that in the first round of
		this new game Avoider is allowed to select as many edges as she wishes (maybe none), while Enforcer has to select
		$r$ edges, for some fixed $r\in\{0,1,\ldots,b\}$. In view of Theorem~\ref{upper_strategy}, analogously to the previous case, we can argue that Enforcer wins the $H$-game on $K_{n'}$ and thereby he wins the $H$-game on $K_n$ as well.
		
		In view of the above three cases, the assertion easily follows.
	\end{proof}

\subsection{Odd unicyclic graphs}
In this section we prove Theorem \ref{thm:odd_cycle}. Our main tool is the following supersaturation-type result. 
\begin{lemma}
\label{lem:threat_count}
Let $H$ be a unicyclic $h$-vertex graph. Then every graph $G$ has at least
$e(G) - (h-2)v(G)$ pairs $\{x,y\} \in \binom{V(G)}{2}$ such that  $E(G) \cup \big\{\{x,y\}\big\}$contains a copy of $H$ in which $\{x,y\}$ is a cycle-edge. 
\end{lemma}
\begin{proof}
	For simplicity, we say that a pair $\{x,y\}$ as in the statement of the lemma is a {\em cycle-threat}.
	Let $n=v(G)$ and $m=e(G)$.
	The proof is by induction on $n$. For $n=1$, the lemma obviously holds since $m - (h-2)n \leq m=0$.
	
	Let $G$ be a graph with $n \geq 2$ vertices and $m$ edges. First assume that 
	$d_G(v) \leq h - 2$ for some $v \in V(G)$ and consider $G' = G \setminus \{v\}$. 
	By the induction hypothesis, the number of cycle-threats in $G'$ (and thus in $G$) is at least
	$e(G') -(h-2)v(G') \geq m - (h-2) - (h-2)(n - 1) = m - (h-2)n$,
	as required. From now on assume that $\delta(G) \geq h-1$. It is well known that under this assumption 
	every embedding of a subtree of any tree $T$ with $h$ vertices
	into $G$ can be extended into an embedding of $T$ into $G$. 
	
	Recall that $H$ consists of a cycle $v_1,\ldots,v_k$ and trees $T_1,\ldots,T_k$ such that $V(T_i)\cap \{v_1,\ldots,v_k\}=\{v_i\}$ for every $i\in [k]$. Consider an embedding 
	$\phi : H \setminus T_{k} \rightarrow G$. Put 
	$U := Im(\phi) \subseteq V(G)$. Then $|U| \leq h-1$.
	We claim that for every vertex
	$y \in N_G \big( \phi(v_{k-1}) \big) \setminus U$, the pair $\{\phi(v_1),y\}$ is a cycle-threat. 
	To prove this, it is enough to extend the embedding $\phi$ by putting $\phi(v_k)=y$ and then to extend it again,
	into an embedding of the tree $E(H) \setminus \{\{v_1,v_k\}\}$ into $G$. 
	
	Similarly, for every vertex $y \in N_G \big( \phi(v_{1}) \big) \setminus U$, the pair $\{\phi(v_{k-1}),y\}$ is a cycle-threat. 
	It follows that the number of cycle-threats that contain one of the vertices $\phi(v_1)$, $\phi(v_{k-1})$ is at least
	\begin{equation}\label{2_vertices}
		d_G \big( \phi(v_1) \big) + d_G \big( \phi(v_{k-1}) \big) - 2(|U|-1) \geq 
		d_G \big( \phi(v_1) \big) + d_G \big( \phi(v_{k-1}) \big) - 2(h - 2),
	\end{equation}
	as $|U| \le h-1$.
	
	Now consider the graph $G'$ obtained from $G$ by removing two vertices $\phi(v_1)$ and $\phi(v_{k-1})$. By the induction hypothesis, the number of cycle-threats in $G'$ is at least
	\begin{equation}\label{recursion}
		e(G') - (h-2)v(G') \geq m - 
		d_G \big( \phi(v_1) \big) - d_G \big( \phi(v_{k-1}) \big) - (h-2)(n-2).
	\end{equation} 
	Adding up (\ref{2_vertices}) and (\ref{recursion}), we get that $G$ contains at least $m - (h-2)n$ cycle-threats, as required. 
\end{proof}

\begin{proof}[\textbf{Proof of Theorem \ref{thm:odd_cycle}}]
	With Fact~\ref{fact:threats2} in mind, we will show that Enforcer has a strategy which guarantees that at some point 
	of the game, either Enforcer has already won or the number of $H$-threats is at least $b+1$. 
	
	At the start of the game Enforcer arbitrarily partitions $V(K_n)$ into two sets $X,Y$ of sizes $\lfloor n/2 \rfloor$ and 
	$\lceil n/2 \rceil$, respectively. Enforcer claims arbitrary edges from $E(X,Y)$ as long as he can. 
	Consider the point in the game just after the last edge in $E(X,Y)$ was claimed.  
	We claim that at this point either Enforcer has already won or there are at least $b+1$ $H$-threats. 
	Let $G_A$ denote Avoider's graph at this point and put $g_{X,Y} = |E(G_A) \cap E(X,Y)|$, $g_{X} = |E(G_A) \cap E(X)|$ and 
	$g_{Y} = |E(G_A) \cap E(Y)|$. Since every edge in $E(X,Y)$ was already claimed, the number of edges claimed by 
	Avoider up to this point is at least 
	$$ e(G_A) \geq \frac{|E(X,Y)| - b}{b+1} \geq \frac{(n^2-1)/4 - b}{b+1} \geq \frac{n^2}{4(b+1)} - 1. $$
	As $e(G_A) = g_{X,Y} + g_{X} + g_{Y}$, one of the numbers 
	$g_{X,Y}, g_X, g_Y$ is not smaller than $n^2/(12(b+1)) - 1$. 
	By our choice of $b$ we have
	$b+1 \leq n/12h \leq n^2/(12(h-1)n + 12)$,
	implying that
	$n^2/(12(b+1)) - 1 \geq (h-1)n$. 
	Assume first that $g_{X,Y} \geq (h-1)n$. 
	Let $G'$ be the graph whose edges are $E(G_A) \cap E(X,Y)$. 
	Then $e(G') = g_{X,Y} \geq (h-1)n$. By Lemma \ref{lem:threat_count}, there are at least 
	$e(G') - (h-2)n \geq n \geq b+1$ pairs of vertices $\{u,v\} \in \binom{V(G')}{2}$ such that $G' \cup \big\{\{u,v\}\big\}$ 
	contains a copy of $H$ in which $\{u,v\}$ is a cycle-edge. Let $\{u,v\}$ be such a pair. Since the cycle in $H$ is odd and $G'$
	only contains edges from $E(X,Y)$, either $u,v \in X$ or $u,v \in Y$. Since Enforcer only claims edges
	between $X$ and $Y$, the edge $\{u,v\}$ is either free or taken by Avoider. If $\{u,v\}$ is taken by Avoider then her graph
	contains a copy of $H$, implying that Enforcer has already won the game. Otherwise, $\{u,v\}$ is an $H$-threat. 
	We conclude that unless Enforcer had already won, there are at least $b + 1$ $H$-threats, as required.
	
	Now assume that $g_X \geq (h-1)n$. In this case let $G'$ be the graph whose edges are $E(G_A) \cap E(X)$. 
	By Lemma \ref{lem:threat_count} and the same argument as in the previous case, either Enforcer had already won or $X$
	contains at least $b+1$ $H$-threats, as required. 
	The case $g_Y \geq (h-1)n$ is handled analogously. This completes the proof of the theorem.   
\end{proof}

\subsection{Even unicyclic graphs}
The aim of this section is to prove Theorem \ref{thm:even_cycle}. For this purpose, we will need the following folklore result. Since we have not been able to find a reference for it, we include a proof here.
\begin{lemma}\label{lem:disj_copies}
Let $T$ be a tree with $t$ vertices. Let $G$ be a graph with maximum degree $\Delta$. Then $G$ contains at least 
$(e(G) - (t-2)v(G))/(t\Delta)$ pairwise vertex-disjoint copies of $T$.  
\end{lemma}
\begin{proof}
	Let $n = v(G)$ and $m = e(G)$. The proof is by induction on $n$. The base case is $n < t$. It is easy to check that in this case 
	$m \leq \binom{n}{2} \leq (t-2)n$, so there is nothing to prove. 
	
	Suppose that 
	$n \geq t$. If $m \leq (t-2)n$, then the lemma holds trivially. Assume, then, that 
	$m > (t-2)n$. It is well known  that if $n \geq t$ and $m > (t-2)n$, then $G$ contains a copy of every tree on $t$ vertices. Let 
	$u_1,\ldots,u_t \in V(G)$ be the vertices of a copy of $T$ in $G$, and let $G'$ be the graph obtained from $G$ by removing the vertices $u_1,\ldots,u_t$. We have $v(G') = n-t$ and $e(G') \geq m - t\Delta$. By the induction hypothesis, $G'$ contains at least 
	$$\frac{e(G') - (t-2)v(G')}{t\Delta} \geq 
	\frac{m - t\Delta - (t-2)n}{t\Delta} \geq 
	\frac{m - (t-2)n}{t\Delta} - 1
	$$ 
	pairwise vertex-disjoint copies of $T$. Now, the copy of $T$ on $\{u_1,\dots,u_t\}$ is clearly disjoint from every copy in $G'$, 
	giving a total of $(m - (t-2)n)/(t\Delta)$ pairwise vertex-disjoint copies of $T$.
\end{proof}

\begin{proof}[\textbf{Proof of Theorem~\ref{thm:even_cycle}}]
	 We will describe a strategy for Enforcer which guarantees that at some point of the game either Enforcer has already won or the number of $H$-threats is at least $b+1$ (implying that Enforcer wins by Fact~\ref{fact:threats2}). 
	We start by defining a tree $T(H)$ which will play an important role in Enforcer's strategy.
	
	Suppose that the unique cycle in $H$ is $v_1,\dots,v_k$. Let $T' = T'(H)$ be the tree obtained from $H$ by contracting the edges of the cycle $v_1,\dots,v_k$ into one vertex $v$.
	Let $T(H)$ be the graph which consists of a path $u_1,\dots,u_{3k/2}$ and $3k/2$ pairwise vertex-disjoint copies of $T'$ in which $u_1,\dots,u_{3k/2}$ play the role of $v$. See Figure 1 for an example of the definition of $T(H)$.  
	Define $t(H) := v(T(H))$.
	It is easy to see that $T(H)$ is a tree and that  
	\begin{equation}\label{eq:t_bound}
	t(H) \leq 3k(h-k+1)/2.
	\end{equation}
	The following observation follows immediately from the definition of $T(H)$. 
	\begin{fact}\label{obs:copy_from_tree}
		Let $P$ and $Q$ be two vertex-disjoint copies of $T(H)$ in a graph $G$, and let $p_1,\dots,p_{3k/2}$ (respectively,
		$q_1,\dots,q_{3k/2}$) be the vertices of $P$ (respectively, $Q$) which play the roles of $u_1,\dots,u_{3k/2}$. 
		If there are 
		$\alpha,\beta,\gamma,\delta \in \{1,\dots,3k/2\}$ for which 
		$\left\{ p_{\alpha},q_{\beta} \right\}, \left\{ p_{\gamma},q_{\delta} \right\} \in E(G)$ and 
		$\left| \alpha - \gamma \right| = \left| \beta - \delta \right| = k/2 - 1$, then $G[P \cup Q]$ contains a copy of $H$. 
	\end{fact} 
	\begin{figure}\label{figure:T(H)}
		\begin{center}
			\begin{tabular}{ c c c} 
				% \hline   
				\begin{tikzpicture}
				\draw[fill] (0.5,0.5) circle [radius=1pt];
				\draw[fill] (0.5,-0.5) circle [radius=1pt];
				\draw[fill] (-0.5,0.5) circle [radius=1pt];
				\draw[fill] (-0.5,-0.5) circle [radius=1pt];
				\draw[fill] (1.5,0.5) circle [radius=1pt];
				\draw[fill] (2.5,0.5) circle [radius=1pt];
				\draw[fill] (-1.5,0.5) circle [radius=1pt];
				\draw[fill] (-0.5,1.5) circle [radius=1pt];
				% \draw[fill] (-0.5,-1.5) circle [radius=1pt]; 
				
				\node at (0,-1) {$H$}; 
				
				\draw (0.5,0.5) -- (0.5,-0.5);
				\draw (0.5,-0.5) -- (-0.5,-0.5);
				\draw (-0.5,-0.5) -- (-0.5,0.5);
				\draw (-0.5,0.5) -- (0.5,0.5);
				\draw (0.5,0.5) -- (1.5,0.5);
				\draw (1.5,0.5) -- (2.5,0.5);
				\draw (-0.5,0.5) -- (-1.5,0.5);
				\draw (-0.5,0.5) -- (-0.5,1.5);
				
				\end{tikzpicture}
				&
				\begin{tikzpicture}
				
				\draw[fill] (0,0) circle [radius=1pt];
				\draw[fill] (1,0) circle [radius=1pt];
				\draw[fill] (2,0) circle [radius=1pt];
				\draw[fill] (0,1) circle [radius=1pt];
				\draw[fill] (-1,0) circle [radius=1pt];
				
				\node at (0.5,-1) {$T'(H)$};
				
				\draw (0,0) -- (1,0);
				\draw (1,0) -- (2,0);
				\draw (0,0) -- (-1,0);
				\draw (0,0) -- (0,1);
				
				\node[below] at (0,0) {$v$};
				
				\end{tikzpicture}
				&
				\begin{tikzpicture}
				
				\draw[fill] (-2,0) circle [radius=1pt]; 
				\draw[fill] (-1,0) circle [radius=1pt]; 
				\draw[fill] (0,0) circle [radius=1pt]; 
				\draw[fill] (1,0) circle [radius=1pt]; 
				\draw[fill] (2,0) circle [radius=1pt]; 
				\draw[fill] (3,0) circle [radius=1pt]; 
				
				\node[below left] at (-2,0) {$u_1$}; 
				\node[below left] at (-1,0) {$u_2$}; 
				\node[below left] at (0,0) {$u_3$}; 
				\node[below left] at (1,0) {$u_4$}; 
				\node[below left] at (2,0) {$u_5$}; 
				\node[below left] at (3,0) {$u_6$};
				
				\draw[fill] (-2,1) circle [radius=1pt]; 
				\draw[fill] (-2,2) circle [radius=1pt]; 
				\draw[fill] (-1,1) circle [radius=1pt]; 
				\draw[fill] (-1,2) circle [radius=1pt]; 
				\draw[fill] (0,1) circle [radius=1pt]; 
				\draw[fill] (0,2) circle [radius=1pt];
				\draw[fill] (1,1) circle [radius=1pt]; 
				\draw[fill] (1,2) circle [radius=1pt]; 
				\draw[fill] (2,1) circle [radius=1pt]; 
				\draw[fill] (2,2) circle [radius=1pt]; 
				\draw[fill] (3,1) circle [radius=1pt]; 
				\draw[fill] (3,2) circle [radius=1pt];
				
				\draw[fill] (-0.195,-0.980) + (-2,0) circle [radius=1pt]; 
				\draw[fill] (0.195,-0.980) + (-2,0) circle [radius=1pt];  
				\draw[fill] (-0.195,-0.980) + (-1,0) circle [radius=1pt]; 
				\draw[fill] (0.195,-0.980) + (-1,0)circle [radius=1pt];  
				\draw[fill] (-0.195,-0.980) + (0,0)circle [radius=1pt]; 
				\draw[fill] (0.195,-0.980) + (0,0)circle [radius=1pt];  
				\draw[fill] (-0.195,-0.980) + (1,0)circle [radius=1pt]; 
				\draw[fill] (0.195,-0.980) + (1,0)circle [radius=1pt];  
				\draw[fill] (-0.195,-0.980) + (2,0)circle [radius=1pt]; 
				\draw[fill] (0.195,-0.980) + (2,0)circle [radius=1pt];  
				\draw[fill] (-0.195,-0.980) + (3,0)circle [radius=1pt]; 
				\draw[fill] (0.195,-0.980) + (3,0)circle [radius=1pt]; 
				
				\draw (-2,0) -- (-1,0);
				\draw (-1,0) -- (0,0);
				\draw (0,0) -- (1,0);
				\draw (1,0) -- (2,0);
				\draw (2,0) -- (3,0);
				
				\draw (-2,0) -- (-2,1);
				\draw (-2,1) -- (-2,2);
				\draw (-1,0) -- (-1,1);
				\draw (-1,1) -- (-1,2);
				\draw (0,0) -- (0,1);
				\draw (0,1) -- (0,2);
				\draw (1,0) -- (1,1);
				\draw (1,1) -- (1,2);
				\draw (2,0) -- (2,1);
				\draw (2,1) -- (2,2);
				\draw (3,0) -- (3,1);
				\draw (3,1) -- (3,2);
				
				\draw (-2,0) -- (-2-0.195,-0.980);
				\draw (-2,0) -- (-2+0.195,-0.980);
				\draw (-1,0) -- (-1-0.195,-0.980);
				\draw (-1,0) -- (-1+0.195,-0.980);
				\draw (0,0) -- (0-0.195,-0.980);
				\draw (0,0) -- (0+0.195,-0.980);
				\draw (1,0) -- (1-0.195,-0.980);
				\draw (1,0) -- (1+0.195,-0.980);
				\draw (2,0) -- (2-0.195,-0.980);
				\draw (2,0) -- (2+0.195,-0.980);
				\draw (3,0) -- (3-0.195,-0.980);
				\draw (3,0) -- (3+0.195,-0.980);
				
				\node at (0.5,-2) {$T(H)$};  
				
				\end{tikzpicture} 
				% \\ \hline
			\end{tabular}
		\end{center}
		
		\caption{The tree $T(H)$ - an example}
	\end{figure}
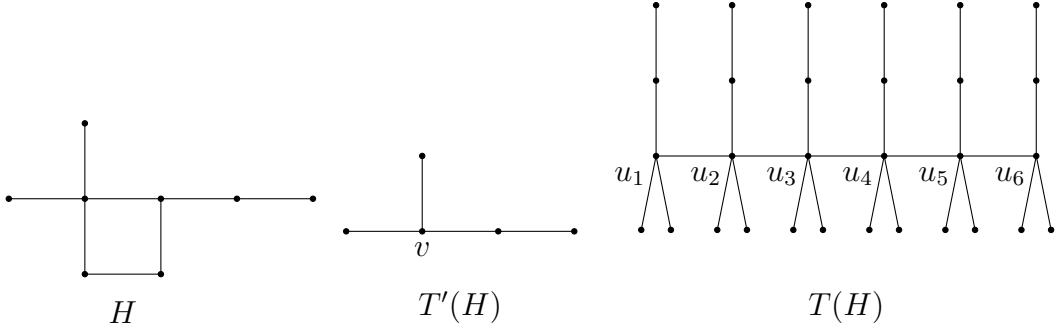
	As in the odd cycle case, Enforcer arbitrarily partitions $V(K_n)$ into two sets $X,Y$ of sizes 
	$\lfloor n/2 \rfloor$ and $\lceil n/2 \rceil$, respectively.  Enforcer's strategy has three stages. Roughly speaking,
	in the first stage Enforcer makes sure that the maximum degree in Avoider's graph inside each of the sets $X$ and $Y$ is not
	large. 
	We will show that unless there are many $H$-threats when the first stage ends, the sets $X$ and $Y$ contain many copies of $T(H)$. The only edges relevant to Enforcer's strategy in later stages are the edges connecting these copies of $T(H)$. Thus, in the second stage Enforcer claims all irrelevant free edges, making sure that they are not free in the third stage.
	This enables Enforcer to force many $H$-threats in the third stage. A precise description of Enforcer's strategy follows. 
	To make the writing shorter, we introduce the following notation. 
	\begin{equation*}
	\St(z) = \begin{cases}
	\{e \in E(X) : z \in e\} & \text{ for }z \in X, \\
	\{e \in E(Y) : z \in e\} & \text{ for }z \in Y.
	\end{cases}
	\end{equation*} 
	
	%\paragraph{Stage I.}
	\paragraph{Stage I}
	The first stage starts at the beginning of the game and ends at the moment when there are no more 
	free edges in $E(X) \cup E(Y)$. The stage may end in the middle of Enforcer's move. 
	In this stage Enforcer only claims edges from $E(X) \cup E(Y)$. 
	In his first round (if he is the first player), Enforcer claims $b$ arbitrary edges from $E(X) \cup E(Y)$.
	Then in every round he responds to Avoider's move as follows. 
	If Avoider claimed an edge in $E(X,Y)$ then Enforcer claims $b$ arbitrary free edges from $E(X) \cup E(Y)$.  
	Suppose now that Avoider claimed an edge in 
	$E(X) \cup E(Y)$, say $\{u,v\}$. Enforcer claims 
	$\lfloor b/2 \rfloor$ free edges from $\St(u)$ and 
	$\lceil b/2 \rceil$ free edges from $\St(v)$ (the roles of $u$ and $v$ are decided arbitrarily). 
	If Enforcer cannot follow this rule (for example, if there are less than $\lfloor b/2 \rfloor$ free edges in $\St(u)$) then Enforcer completes his move by claiming additional arbitrary edges from $E(X) \cup E(Y)$. 
	
	It is clear that Enforcer can follow his strategy for the first stage. 
	We now show that by following his strategy, Enforcer
	succeeds in keeping the maximum degree in $G_A[X]$ and $G_A[Y]$ not too high.  
	
	\begin{claim}\label{claim:max_degree}
		Throughout the game it holds that 
		$\Delta\left( G_{A}[X] \right) \leq 3|X|/(b+1)$ and 
		$\Delta\left( G_{A}[Y] \right) \leq 3|Y|/(b+1)$.   
	\end{claim}
	\begin{proof}
		We only prove the statement for $X$, as the proof for $Y$ is similar.  
		Let $x \in X$ and assume that $d_{G_{A}[X]}(x) = d$ at the end of the first stage. By Enforcer's strategy, every time Avoider claimed an edge from $\St(x)$, Enforcer immediately claimed at least $\lfloor b/2 \rfloor$ such edges, unless there were less than $\lfloor b/2 \rfloor$ free edges remaining in $\St(x)$, which could only happen at the last time Avoider claimed an edge from $\St(x)$. Therefore Enforcer has claimed at least $\lfloor b/2 \rfloor (d-1)$ edges from $\St(x)$. Since all edges in $E(X)$ have been claimed by the end of the first stage, we have
		$d + \left\lfloor \frac{b}{2} \right\rfloor \cdot (d-1) \leq 
		\left| \St(x) \right| = |X| - 1,$ implying that
		$d \leq 
		(|X| - 1 + \left\lfloor \frac{b}{2} \right\rfloor)/
		(\left\lfloor \frac{b}{2} \right\rfloor + 1) 
		\leq
		2|X|/(b+1) + 1 \leq 
		3|X|/(b+1)$. 
	\end{proof} 
	
	Now we prove three  claims about the position at the end of the first stage. 
	Put $g_{X} = |E(G_A) \cap E(X)|$, $g_{Y} = |E(G_A) \cap E(Y)|$ and 
	$g_{X,Y} = |E(G_A) \cap E(X,Y)|$. 
	
	\begin{claim}\label{claim:many_bipartite_edges}
		If $g_{X,Y} \geq (h-1)n$ then either Enforcer has already won or there are at least $b+1$ $H$-threats. 
	\end{claim}
	\begin{proof}
		Let $G'$ be the graph whose edges are $E(G_A) \cap E(X,Y)$. 
		Then $e(G') = g_{X,Y} \geq (h-1)n$. By Lemma \ref{lem:threat_count}, there are at least $e(G') - (h-2)n \geq n \geq b+1$ 
		pairs of vertices $\{u,v\} \in \binom{V(G')}{2}$ such that 
		$G' \cup \big\{\{u,v\}\big\}$ contains a copy of $H$ in which $\{u,v\}$ is a cycle-edge. Let $\{u,v\}$ be such a pair. Since the cycle in $H$ is even and $G'$ only contains edges between $X$ and $Y$, we must have that either 
		$u \in X$ and $v \in Y$ or vice versa. Since in the first stage Enforcer only claimed edges from $E(X) \cup E(Y)$, the edge $\{u,v\}$ is either free or has been taken by Avoider. In the latter case, Avoider's graph contains a copy of $H$, implying that Enforcer has already won. In the former case, $\{u,v\}$ is an $H$-threat. We conclude that unless Enforcer has already won, there are at least $b + 1$ $H$-threats, as required.
	\end{proof}
	
	\begin{claim}\label{claim:edges_in_sets}
		If $g_{X,Y} < (h-1)n$ then $g_X > 25t(H)|X|$ and 
		$g_Y > 25t(H)|Y|$. 
	\end{claim}
	\begin{proof}
		We only prove the statement for $g_X$, as the proof for $g_Y$ is symmetric.
		According to Enforcer's strategy in the first stage, if Enforcer claims an edge from $E(X)$ in some round 
		(other than %her 
		{his} first, if he is the first player), then Avoider's preceding move must have been one of the following types: 
		\begin{enumerate}
			\item Avoider claimed an edge from $E(X)$.
			\item Avoider claimed an edge from $E(X,Y)$.
			\item Avoider claimed an edge $\{u,v\} \in E(Y)$ such that after this Avoider's move there were less than $\lceil  b/2 \rceil$ free edges in $\St(u)$ or less than $\lceil b/2 \rceil$ free edges in $\St(v)$. 
		\end{enumerate}
		For $i\in \{1,2,3\}$, let $a_i$ be the number of edges in $E(X)$ that Enforcer claimed after a move of type $i$ by Avoider. 
		It is clear that $a_1 \leq b \cdot g_X$ and $a_2 \leq b \cdot g_{X,Y}$. 
		Let $e_1,\dots,e_{\ell}$ be the edges claimed by Avoider in moves of type $3$. For $i \in [\ell]$, let $y_i$ be an endpoint of $e_i$ 
		such that after Avoider claimed $e_i$, there remained less than 
		$\lceil b/2 \rceil$ free edges in $\St(y_i)$. We claim that $y_i \neq y_j$ for every $1 \leq i < j \leq \ell$. 
		Suppose by contradiction that $y_i = y_j$ for some $i < j$. 
		After Avoider claimed $e_i$, Enforcer immediately claimed all remaining free edges in $\St(y_i)$, implying that Avoider could not
		have claimed $e_j$ later in the game, a contradiction. Since $y_1,\dots,y_{\ell}$ are distinct we get that
		$\ell \leq |Y| \leq \left\lceil n/2 \right\rceil$. Therefore
		$a_3 \leq \ell \cdot b \leq (n+1) b/2$.
		Observe that 
		$b + a_1 + a_2 + a_3 + g_{X} \ge \binom{|X|}{2}$,
		as the left hand side counts the number of edges in $E(X)$ claimed by either Avoider or Enforcer when Enforcer is the first player.  
		By this inequality and our assumption that 
		$g_{X,Y} < (h-1)n$, we get that
		\begin{equation*}
		\binom{|X|}{2} \le b + a_1 + a_2 + a_3 + g_X < b + (b+1)g_X + (h-1)nb + (n+1)b/2.
		\end{equation*}
		This implies 
		\begin{equation*}
		g_X > \frac{\binom{|X|}{2} - hnb}{b+1} \geq
		\frac{n|X|}{5(b+1)} \geq   
		25 \cdot \tfrac{3}{2}k(h - k + 1)|X| \geq 
		25t(H)|X|,
		\end{equation*}
		where we used $|X| \geq (n-1)/2$, (\ref{eq:t_bound}) and 
		$b + 1 \leq n/(200k(h - k + 1))$ which holds by assumption. 
	\end{proof}
	
	\begin{claim}\label{claim:tcopies}
		Suppose that there are less than $b+1$ $H$-threats at the end of Stage I. Then there are at least $8(b+1)$ pairwise disjoint copies of $T(H)$ in $G_A[X]$ and at least $8(b+1)$ pairwise disjoint copies of $T(H)$ in $G_A[Y]$.
	\end{claim}
	\begin{proof}
		Suppose that there are less than $b+1$ $H$-threats in $G_A$.
		By combining Claims \ref{claim:max_degree}, {\ref{claim:many_bipartite_edges}} and \ref{claim:edges_in_sets} we get that 
		$\Delta\left( G_A[X] \right) \leq 3|X|/(b+1)$
		and
		%$e\left( G_A[X] \right) >g_X \geq 25t(H)|X|$.
		$e\left( G_A[X] \right) = g_X \geq 25t(H)|X|$.
		By Lemma \ref{lem:disj_copies}, $G_A[X]$ contains at least
		\begin{equation*}
		\frac{25t(H)|X| - (t(H)-2)|X|}{t(H)\cdot 3|X|/(b+1)} \geq 
		8(b+1)
		\end{equation*}
		pairwise disjoint copies of $T(H)$. 
		The proof for $G_A[Y]$ is similar. 
	\end{proof}
	
	\paragraph{Stage II}
	We may assume that at the beginning of the second stage there are less than $b+1$ $H$-threats (because otherwise Enforcer had already achieved his goal). Fixing
	\begin{equation*}
	r = 8(b+1),
	\end{equation*}
	Enforcer identifies $r$ pairwise vertex-disjoint copies of $T(H)$ in $G_A[X]$, denoted by $P_1,\ldots,P_r$, and $r$ pairwise vertex-disjoint copies of $T(H)$ in $G_A[Y]$, denoted by $Q_1,\ldots,Q_r$. This is possible in view of Claim~\ref{claim:tcopies}.
	For each $1 \leq i \leq r$, let $p^i_{1},\dots,p^i_{3k/2}$ (respectively, $q^i_{1},\dots,q^i_{3k/2}$) be the vertices of $P_i$ (respectively, $Q_i$) which play the roles of $u_1,\dots,u_{3k/2}$ (recall the definition of $T(H)$). 
	Define 
	\begin{equation*}
	E^* = \left\{ \{p^i_{\alpha},q^j_{\beta} \} : 
	1 \leq \alpha,\beta \leq 3k/2, 1 \leq i,j \leq r \right\}
	\end{equation*}
	and
	\begin{equation*}
	E^{**} = \left\{ \{p^i_{\alpha},q^j_{\beta} \} : 
	1 \leq \alpha,\beta < k/2, 1 \leq i,j \leq r \right\}.
	\end{equation*}
	In the second stage Enforcer claims arbitrary free edges {\em not} belonging to $E^*$. The second stage ends at the moment when all remaining free edges are in $E^*$.   
	
	\paragraph{Stage III}
	In the third stage Enforcer claims arbitrary edges from $E^{**}$. 
	The third stage ends at the moment that every edge from $E^{**}$ has been claimed. 
	
	\bigskip
	It is clear that Enforcer can follow his strategy for Stages II and III. 
	Let us analyse the position at the end of Stage III.
	Let $\mathcal{F}$ be the collection of all pairs 
	$(i,j) \in [r]^2$ such that Avoider claimed some edge in 
	$\left\{ \{p^i_{\alpha},q^j_{\beta}\} : 1 \leq \alpha,\beta \leq 3k/2 \right\}$. Let $(i,j) \in \mathcal{F}$ and suppose that Avoider claimed the edge
	$\{p^i_{\alpha},q^j_{\beta}\}$. It is easy to see that there are $\gamma$ and $\delta$ such that
	$k/2\leq \gamma \leq 3k/2$, $k/2\leq \delta \leq 3k/2$ and  
	$|\alpha - \gamma| = |\beta - \delta| = k/2 - 1$. 
	By Fact~\ref{obs:copy_from_tree}, adding the edge $\{p^i_{\gamma},q^j_{\delta}\}$ to $G_A$ would create a copy of $H$ in $G_A[P_i \cup Q_j]$. 
	Since $\gamma,\delta \geq k/2-1$ we have 
	$\{p^i_{\gamma},q^j_{\delta} \} \notin E^{**}$. Since in the third stage Enforcer only claimed edges in $E^{**}$, the edge
	$\{p^i_{\gamma},q^j_{\delta} \}$ is either free or has been claimed by Avoider. In the latter case Avoider's graph contains
	a copy of $H$, implying that Enforcer has already won. In the former case, $\{p^i_{\gamma},q^j_{\delta} \}$ is an $H$-threat.
	We conclude that either Enforcer has already won or the number of $H$-threats is at least $\left| \mathcal{F} \right|$. 
	
	It remains to show that $\left| \mathcal{F} \right| \geq b + 1$. 
	Let $a$ be the number of edges that Avoider claimed in the third stage, and note that the number of Enforcer's moves in the third stage is at most $a+1$. Let $m$ be the number of edges in $E^{*}$ that Avoider had claimed in the first and second stages. Then by the end of the third stage Avoider has claimed $a+m$ edges in $E^{*}$. From the definition of $E^{*}$, we see that there are at least $(a+m)/(3k/2)^2$ pairs $(i,j) \in [r]^2$ such that Avoider claimed some edge in 
	$\left\{ \{p^i_{\alpha},q^j_{\beta}\} : 1 \leq \alpha,\beta \leq 3k/2 \right\}$, that is, 
	$|\mathcal{F}| \geq (a+m)/(3k/2)^2$.
	Recall that Enforcer had not claimed any edge in $E^{*}$ in the first and second stages. Since by the end of the third stage every edge in $E^{**}$ has been claimed, we have
	$a + (a+1)b \geq \left| E^{**} \right| - m = 
	r^2 \cdot \left( k/2 - 1 \right)^2 - m \geq r^2k^2/16 - m,$
	implying that
	$a \geq (r^2k^2/16-m-b)/(b+1)$. By our choice of $r$ we get that
	\begin{equation*}
	\left| \mathcal{F} \right| \geq \frac{a+m}{( 3k/2 )^2} \geq 
	\frac{r^2k^2/16-m-b}{( 3k/2 )^2(b+1)} + \frac{m}{( 3k/2 )^2} \geq 
	\frac{r^2k^2/16-b}{( 3k/2 )^2(b+1)} =
	\frac{4k^2(b+1)^2-b}{( 3k/2 )^2(b+1)} \ge b+1,
	\end{equation*} 
	as required. This completes the proof of the theorem. 
\end{proof}
\noindent

\section{Games on blow-ups of a multigraph}\label{sec:blowup}

In this section, we present the proof of Theorem \ref{upper_strategy}. 
For this purpose we consider some auxiliary games, played not on the complete graph but on blow-ups of $H$. 
In the auxiliary games, in some rounds the players can select more or less edges than in other rounds. 
Nonetheless, with a slight abuse of terminology, 
we still call them Avoider--Enforcer games.
 
\subsection{The general setup}\label{sec:setup}

For a graph $H$ on the vertex set $[h]$
 we define a graph $\B_H(V_1, V_2,\ldots, V_h)$, called a \emph{blow-up of $H$}, as follows. 
We replace every vertex $i$ of $H$ with an independent set $V_i$ and every edge $\{i,j\}$ of $H$
 with the corresponding complete bipartite graph, that is, with the set of edges 
 $\{xy : x \in V_i, y \in V_j\}$. 

We will also consider blow-ups of connected multigraphs $H$ containing exactly one cycle, which is either a loop $C_1$, 
or a $C_2$ (the multigraph with two vertices joined by two edges). 
With a little abuse of terminology, we extend the definition of a unicyclic graph to include such multigraphs.
We define the blow-up $\B_{H}(V_1, V_2, \ldots, V_h)$ of such a multigraph similarly to the blow-up of a simple graph. Note that if $C_1\subs H$, say if $i \in V(H)$ has a loop, then all vertices in $V_i$ have loops. Similarly, if $C_2\subs H$, say if $i,j \in V(H)$ are joined by two edges, then 
$E(V_i,V_j)$ is a multiset such that every pair $x\in V_i$, $y\in V_j$ is joined by two edges. 
In this case, a couple of edges $\{e,f\}$ in $E(V_i,V_j)$ with the same ends will be called \emph{friends}. 
In an Avoider--Enforcer $H$-game on $\B_{H}(V_1, V_2, \ldots, V_h)$ the players select edges from the multiset
$E(\B_{H}(V_1, V_2,\ldots,V_h))$ and Avoider tries to avoid a copy of $H$ in her multigraph.

Let $H'$ be a submultigraph of a multigraph $H$ and let $V(H')=\{i_1, i_2,\ldots, i_r\}$. 
A copy of $H'$ in ${\mathbb B}_H(V_1, V_2,\ldots, V_h)$ is said to be \emph{canonical} if {the vertex playing the role of $i_j$ in this copy is from $V_{i_j}$} 
(for every $1 \leq j \leq r$).  In the Avoider--Enforcer $H$-game played on 
${\mathbb B}_H(V_1,\ldots,V_h)$,
we call an edge $e \in E$ a \emph{canonical $H$-threat} or simply a \emph{canonical threat} 
if $e$ is free and there exists a canonical $H$-copy in $G_A \cup \{e\}$ that is not contained in $G_A$.

Throughout Section \ref{sec:blowup} we assume that 
{$h\ge 2$, $H$ is either a tree with the vertex set $V(H)=[h]$, or it is a unicyclic multigraph with the vertex set 
$V(H)=[h]$ and (unique) cycle $C_k$ ($k\ge 1$) on vertices $1,\dots,k$.}
Let $b$ and $r$ be two integers such that $0\le r\le b$. 
We will only consider the $(1:b)$ Avoider--Enforcer game played on 
${\mathbb B}_{H}(V_1,\ldots,V_h)$, in which 
\begin{equation}\label{bigprod_match}
4^{h} \le |V_i| \le b/(8h) \quad \text{for every $i\in [h]$, and} \ 
\prod_{i=1}^h|V_i|>
\begin{cases}
(16^h (b+1)h)^{h-1}, &\text{ if $H$ has a cycle,}\\
{(16^h (b+1)h)^{h-2}\max_j |V_j|},&{\text{ if $H$ is a tree.}}
\end{cases}
\end{equation}
Additionally, assume that in the first round (and only then)
Avoider can select any number of edges
of ${\mathbb B}_{H}(V_1,\ldots,V_h)$ she wishes and Enforcer has to select exactly $r$ edges.

 % % % % % % % % % % % % % C2

\subsection{Proof of Theorem~\ref{upper_strategy}}

In this section we give the proof of Theorem~\ref{upper_strategy}, modulo Proposition \ref{blowup_unic}, stated below, which says that Enforcer can force many $H$-threats in the game played on the blow-up of $H$. The proof of Proposition \ref{blowup_unic} appears in the next section.

\begin{prop}\label{blowup_unic}
{Assume the setting as introduced in Section~\ref{sec:setup}}. Then Enforcer has a strategy in which, at some point, 
	either he has already won or the number of canonical $H$-threats is at least
	$$
	\frac{\prod_{i=1}^h |V_i|}{(8^{h}(b+1)h)^{{e(H)}-1}}\,.
	$$
\end{prop}

Now, we deduce Theorem~\ref{upper_strategy} from the above result.

\begin{proof}[\textbf{Proof of Theorem~\ref{upper_strategy}}]
	{Let $V(K_n)=V_1\sqcup V_2 \sqcup \ldots \sqcup V_h$} be a partition of the vertex set of $K_n$
	such that $\lf n/h\rf\le |V_i|\le \lc n/h\rc$ for every $i\in [h]$.
	Consider the blow-up ${\mathbb B}_{H}(V_1,V_2,\ldots,V_h)\subs K_n$. 
	Set $\gamma=\frac{1}{16^hh (2h)^{e(H)/(e(H)-1)}}$, and suppose that $b$ is an even number with 
	\begin{equation}\label{bbound}
	8h n\le b+1\le \gamma n^{{e(H)/(e(H)-1)}}\,.
	\end{equation}
	We split Enforcer's strategy into two stages.
	
	\paragraph{Stage I}
	Enforcer picks all free edges that are not edges of the blow-up 
	${\mathbb B}_{H}(V_1,V_2,\ldots,V_h)$, until no such free edges remain.
	This stage might end in the middle of Enforcer's move. 
	
	\paragraph{Stage II}
	In this stage the game is played on ${\mathbb B}_{H}(V_1,V_2,\ldots,V_h)$. 
	At the beginning some edges are already selected by Avoider but there is no edge of Enforcer on the board.
	Let $r\in \{0,1,\ldots,b\}$ be the number of edges Enforcer has to select in order to complete his last move of Stage I.
	Enforcer then uses the strategy from Proposition~\ref{blowup_unic}. 
	Let us verify that this is possible.
	Observe that \eqref{bbound} implies
	$4^{h} \le |V_i| \le \frac{b}{8h}$ for every $1 \le i \le h$.
	Furthermore, by our choice of $\gamma$ and \eqref{bbound}, {in case of $H$ with a cycle} we have
	$$
	\prod_{i=1}^h |V_i|>\Big(\frac{n}{2h}\Big)^h \ge (16^{h} (b+1)h)^{h-1}\,,
	$$
	{and in case $H$ is a tree we have
		$$
		\prod_{i=1}^h |V_i|>\Big(\frac{n}{2h}\Big)^{h-1}\max_j |V_j| \ge (16^{h} (b+1)h)^{h-2}\max_j |V_j|\,.
		$$}
	Therefore we can indeed apply Proposition~\ref{blowup_unic}. As a consequence, Enforcer has a strategy in which at some point, either he has already won or the number of threats is not less than
\begin{equation*}
	\frac{\prod_{i=1}^h |V_i|}{(8^{h}(b+1)h)^{{ e(H)}-1}}>\frac{\gamma^{e(H)-1} n^h}{(b+1)^{{e(H)}-1}}. \qedhere
\end{equation*}
\end{proof}

\subsection{Main lemmata and the proof of Proposition \ref{blowup_unic}}
\label{sec:prop-threats}

The idea behind Enforcer's strategy is quite simple: he tries to force Avoider to create many trees and then to extend them gradually, till a copy of $H$ is completed. He can achieve this by following an inductive strategy. The base case is captured by the following statement, whose proof is given in Section \ref{sec:base-case}.

\begin{lemma}[{\bf Base case}]\label{blowup_triangleM}
{Assume the setting as introduced in Section~\ref{sec:setup}}. Suppose further that $b$ is even and $H=C_2$. Then Enforcer has a strategy in which, at some point, either he has already won or the number of $C_2$-threats is not less than
$$\frac{|V_1||V_2|}{4(b+1)}.$$
\end{lemma}

The following statement, to be proven in Section \ref{sec:inductive-step}, allows Enforcer to reduce a game played on a blow-up of $H$ to a game on a blow-up of a proper sub(multi)graph of $H$. 

\begin{lemma}[{\bf Inductive step}]\label{blowup_matching}
{Assume the setting as introduced in Section~\ref{sec:setup}}. Suppose further that $h\ge 3$. Then Enforcer has a strategy which guarantees that, at some point, there exist sets $V'_i\subs V_i$ for $1 \le i  \le h$, a subgraph $F_s\subs H$ on $s$ vertices,   
and a family $\F$ of vertex-disjoint canonical copies of $F_s$ in Avoider's graph, 
such that the following conditions are satisfied.
\begin{enumerate}[{\rm (P1)}]
\item\label{vprime}
$V'_i=V_i\cap \bigcup_{F\in\F} V(F)$ 
for every $i\in V(F_s)$.
\item\label{bigsets}
$|V'_i|\ge |V_i|/4$ for every $i\notin V(F_s)$.
\item\label{no_enf}
Enforcer has claimed no edges which are incident to vertices in 
$\bigcup_{i\in V(H)\setminus V(F_s)}V'_i$.
\item\label{no_enf_threat}
For every  $F\in\F$ there is no Enforcer's edge with both endpoints in $V(F)$. 
\item\label{manypaths}
Either $F_s$ is a loop and $|\F|\ge \max\limits_{ij \in E(H), i \ne j}\frac{|V_i||V_j|}{64(b+1)h}$, or $F_s$ is 
a non-trivial tree  and ${|\F|\ge \frac{\prod_{i\in V(F_s)} |V_i|}{(64(b+1)h)^{s-1}}}$.
\end{enumerate}
\end{lemma}

\begin{proof}[\textbf{Proof of Proposition~\ref{blowup_unic}}]
We will use induction on $h$. 
{If $H$ is a tree and $h=2$, then the assertion easily follows.
If $H$ has a cycle and $h=k=2$}, 
the assertion follows immediately from Lemma~\ref{blowup_triangleM}.
We next consider the case {of a unicyclic $H$ with} $(h,k)=(2,1)$. Let $\F$  be the family given by Lemma~\ref{blowup_matching}. According to property (P\ref{manypaths}), 
$\F$ is a matching or a set of loops of size at least $|V_1||V_2|/(64(b+1)h)$. 
From properties (P\ref{bigsets})--(P\ref{no_enf_threat}), it follows that
Avoider's graph contains either a copy of $H$, or at least $|V_1||V_2|/(64(b+1)h)=|V_1||V_2|/(8^h(b+1)h)$ 
canonical $H$-threats.

Let $h\ge 3$ and suppose that the assertion is true for {every tree and} every unicyclic multigraph on less than $h$ vertices.
We will show that, roughly speaking, Enforcer can force many vertex-disjoint canonical trees in Avoider's graph.
After contracting some edges of the trees he then can proceed inductively. 
Enforcer's strategy is divided into three stages. In the first he forces many trees,
in the second he selects edges irrelevant for later play, and in the third he applies the inductive argument.

\paragraph{Stage I}
Due to Lemma~\ref{blowup_matching},  
Enforcer has a strategy so that at some point there exist vertex-sets 
$V'_i\subs V_i$ for $1 \le i \le h$, 
and a family $\F$ of vertex-disjoint canonical copies of some $F_s\subs H$ in Avoider's graph, which satisfy properties (P\ref{vprime})--(P\ref{manypaths}). 
Suppose first that $F_s$ is a tree on $h$ vertices. {If $F_s = H$ then Enforcer had already won (as there is a copy of $F_s = H$ in Avoider's graph), so we assume from now on that $F_s \neq H$, namely that $H$ is unicyclic. Then $e(H) = h = s$}. In view of (P\ref{manypaths}), we have
$$|\F|\ge \frac{\prod_{i=1}^h |V_i|}{(64 (b+1)h)^{{ e(H)}-1}}>\frac{\prod_{i=1}^h |V_i|}{(8^h(b+1)h)^{{ e(H)}-1}}\,.$$
If there is no copy of $H$ in Avoider's graph yet, then every tree in $\F$ is a canonical threat, by (P\ref{no_enf_threat}).
Thus the proposition follows, and we do not have to proceed to Stage II.

From now on we assume that $F_s$ is not a tree on $h$ vertices, that is, $V(F_s)\subsetneq V(H)$. 
The game proceeds to the next stage.

\paragraph{Stage II}
If $F_s$ is a tree, we say that an edge $e\in E(V'_i,V'_j)$ is {\em important} 
if $\{i,j\}\cap (V(H)\setminus V(F_s)) \ne \emptyset$, or
$ij\in E(C_k)\setm E(F_s)$ and the ends of $e$ belong to the same tree in $\F$ ($e$ may be a loop).
 
If $F_s$ is a loop, we define important edges differently.  
Suppose that $x\in V(H)$ is the only vertex of $H$ having a loop, 
and fix an arbitrary $y \in V(H) \setminus \{x\}$ such that $xy\in E(H)$ (such a $y$ exists because $H$ is connected). 
Let $\cM$ be a matching in $E(V'_x,V'_y)$ of size at least $|V_x||V_y|/(256(b+1)h)$. 
Such a matching exists because 
$|V'_x|=|\F|\ge |V_x||V_y|/(64(b+1)h)$ by properties (P\ref{vprime}) and (P\ref{manypaths}),
and $|V'_y|\ge |V_y|/4> |V_x||V_y|/(64(b+1)h)$ due to (P\ref{bigsets}) and \eqref{bigprod_match}. 
Let $V''_x=V'_x\cap V(\cM)$,  $V''_y=V'_y\cap V(\cM)$, and $V''_i=V'_i$ for every $i\in V(H) \setm \{x,y\}$.   
We call all edges in $\cM\cup\bigcup_{ij\in E(H)\setm \{xy\}}E(V''_i,V''_j)$ {\em important}. 

In Stage II Enforcer selects all free edges that are {\it not} important.  
The stage might end in the middle of Enforcer's move. 

\paragraph{Stage III}
At the beginning of this stage, the set of important edges contains every free edge, and
in view of (P\ref{no_enf}), (P\ref{no_enf_threat}) and Enforcer's strategy in the second stage, Enforcer has claimed no edges in this set. 
We are going to reduce the target graph $H$ and for this purpose we consider two cases.

\begin{description}
\item[\bf Case 1] $F_s$ is a tree.

Without loss of generality we can assume that $V(F_s)=[s]$
(here we make an exception to the convention that $1,\dots,k$ are the vertices of the {unique} cycle in $H$, {if there is one}).
We contract every tree in $\F$ to a vertex, 
and consider a blow-up $\B_{H'}(V'_s,V'_{s+1},\ldots,V'_h)$ of the graph $H'$ obtained from $H$ by contracting $F_s$ to the vertex $s$. 
Note that if $F_s$ contains an edge $e$
of a $C_2$, then the friend of $e$ becomes a loop in $H'$.  
{Observe that $H'$ is unicyclic if $H$ is unicyclic, and $H'$ is a tree if $H$ is a tree.}  
Every edge of $\B_{H'}(V'_s,V'_{s+1},\ldots,V'_h)$ is either selected by Avoider
or it was an important edge in $\B_{H}(V'_s,V'_{s+1},\ldots,V'_h)$ and thus it has not been selected by Enforcer.

To prove the proposition, it suffices to show that Enforcer has a strategy
in the $(1:b)$ $H'$-game played on $\B_{H'}(V'_s,V'_{s+1},\ldots,V'_h)$, 
which forces either a copy of $H'$ in Avoider's graph or at least 
$$
\frac{\prod_{i=1}^h |V_i|}{(8^{h}(b+1)h)^{{e(H)}-1}}
$$
canonical $H'$-threats.

\item[\bf Case 2] $F_s$ is a loop.

We can assume that the vertices $x$ and $y$ of $H$ considered in 
Stage II are equal to $1$ and $2$, respectively. Recall that the edge-set $\cM$ defined in Stage II 
is a perfect matching in $E(V''_1,V''_2)$.
We glue together the ends of every edge $e$ in $\cM$ (turning $e$ into a loop), delete all the loops in $\F$,
and consider a blow-up $\B_{H'}(V''_2,V''_3,\ldots,V''_h)$ of $H'$, which is the unicyclic graph obtained from $H$ by contracting the edge $\{1,2\}$ to a vertex. 
As in Case 1, to prove the proposition, it is enough to show that Enforcer has a strategy in the $(1:b)$ $H'$-game played on $\B_{H'}(V''_2,V''_3,\ldots,V''_h)$, 
which guarantees either a copy of $H'$ in Avoider's graph or at least 
$\frac{\prod_{i=1}^h |V_i|}{(8^{h}(b+1)h)^{{e(H)}-1}}$ canonical $H'$-threats.
\end{description}

We will prove that the induction hypothesis can be applied to $H'$ in both cases.  
We only consider Case 1, as the argument for Case 2 is the same as for the special case of Case 1 in which 
$H$ has a cycle, $s=k=2$ and the $C_2$ in $H$ coincides with the edge being contracted. 

So our goal now is to show that \eqref{bigprod_match} holds for $H'$ and the sets $V'_s,\dots,V'_h$. 
Since \eqref{bigprod_match} holds for $H$, we have $|V'_i| \le |V_i| \le
b/(8v(H)) \le b/(8v(H'))$ for every $i \in \{s,\ldots,h\}$. 
From (P\ref{bigsets}) and \eqref{bigprod_match}, we know that $|V'_i|\ge |V_i|/4\ge 4^{h-1} \geq {4^{v(H')}}$ for $i\in \{s+1,\ldots,h\}$. By (P\ref{manypaths}), \eqref{bigprod_match} and $h \geq 3$, we have
$$
|V'_s|= |\F|\ge \frac{\prod_{i=1}^s |V_i|}{(64(b+1)h)^{s-1}}>
{\frac{\prod_{i=1}^h |V_i|}{(64h)^{s-1}(b+1)^{h-2}|V_h|} \ge 
\frac{(16^{h}(b+1)h)^{h-2}\max_j{|V_j|}}{(64h)^{s-1}(b+1)^{h-2}|V_h|} } >
4^{h-1} \ge 4^{v(H')}.
$$

Finally, as $|V'_i|\ge |V_i|/4$ for every $s+1 \le i\le h$, and as 
$|V'_s|=|\F|\ge \frac{\prod_{i=1}^s |V_i|}{(64 (b+1)h)^{s-1}}$, $v(H') = h - s + 1$ and $h \geq 3$, in the case that $H$ is unicyclic it holds that  
\begin{eqnarray*}
\prod_{i=s}^{h} |V'_i|&\ge & 
\frac{\prod_{i=1}^h |V_i|}{(64 (b+1)h)^{s-1}4^{h-s}} \overset{\eqref{bigprod_match}}{>}
\frac{(16^h(b+1)h)^{h-1}}{(64(b+1)h)^{s-1}4^{h-s}} >
(16^{v(H')}(b+1)v(H'))^{v(H')-1},
\end{eqnarray*}
and in the case that $H$ is a tree it holds that
\begin{eqnarray*}
\prod_{i=s}^{h} |V'_i|&\ge & 
\frac{\prod_{i=1}^h |V_i|}{(64 (b+1)h)^{s-1}4^{h-s}} \overset{\eqref{bigprod_match}}{>}
\frac{(16^h(b+1)h)^{h-2}\max_j|V_j|}{(64(b+1)h)^{s-1}4^{h-s}} >
(16^{v(H')}(b+1)v(H'))^{v(H')-2}\max_j|V_j|.
\end{eqnarray*}

{Thus, we have proven that the induction hypothesis can be applied to $H'$.} By the induction hypothesis Enforcer has a strategy such that at some point, 
either there is a canonical copy of $H'$ in Avoider's graph, or the number of canonical $H'$-threats is at least    
\begin{align*}
\frac{\prod_{i=s}^h |V'_i|}{(8^{h-s+1}(h-s+1)(b+1))^{{e(H')-1}}}
&\ge
\frac{\prod_{i=1}^h |V_i|}{(8^{h-s+1}(h-s+1)(b+1))^{{e(H')-1}}\cdot (64(b+1)h)^{s-1}4^{h-s}}\\
&>
\frac{\prod_{i=1}^h |V_i|}{(8^{h} h(b+1))^{{e(H)}-1}}.\qedhere
\end{align*}
\end{proof}

\subsection{The base case}\label{sec:base-case}
Here we provide a self-contained proof for the base case.

\begin{proof}[\textbf{Proof of Lemma~\ref{blowup_triangleM}}] %triangleM
	Suppose that after the first move of Avoider there is no copy of $C_2$ in her multigraph.
	Let $X$ be the set of all vertices of $V_1$ isolated in Avoider's multigraph after her first move. If $|X|\le |V_1|/2$, then the number of threats is at least
	$$ |V_1\setminus X| \ge |V_1|/2 \overset{\eqref{bigprod_match}}{>}|V_1||V_2|/(2b),$$
	and the assertion follows.
	From now on we assume that $|X|>|V_1|/2$. Let  
	$$T=\frac{|V_1||V_2|}{4(b+1)}\,.$$
	From the definition of $X$ we see that
	every edge of the multiset $E(X,V_2)$ is free. Furthermore, we have
	$
	{|E(X,V_2)|=2|X||V_2| > |V_1| \cdot |V_2|
		=4(b+1)T}
	$. Thus the number of rounds in the game is greater than
	$4T$. Note that $T>1$, by \eqref{bigprod_match}.
	Furthermore there are $|E(X,V_2)|/2>2(b+1)T$  
	free couples of friends.
	
	Enforcer plays as follows. In the first round he begins by picking $\lf r/2\rf$ free couples of friends. 
	After that he has either zero edges 
	or one edge left to take, depending on the parity of $r$.  
	If $r$ is odd, Enforcer finishes the move by picking an arbitrary edge in $E(X,V_2)$, say $e$. 
	Denote the friend of $e$ by $f$ 
	and note that $f$ 
	is the only free edge in $E(X,Y)$ whose friend has been
	selected by Enforcer.
	Any other free edge is either a threat or has a free friend. 
	If $r$ is even, then every  free edge is either a threat or has a free friend.
	Avoider in her response must create a new threat or, if $r$ is odd, select $f$ 
	(we assume that Avoider never selects a threat, as this would lose the game).
	
	In the next $2\lf T\rf$ rounds, Enforcer always picks $b/2$  free couples of friends. 
	Recall that $b$ is even and, as mentioned above, there are enough free couples to play in this way.
	In every round Avoider will create a new threat, with only one possible exception, that is, when she selects $f$.
	We conclude that the number of threats created in the game is at least 
	$2\lf T\rf>T$.
\end{proof}

\subsection{The inductive step}\label{sec:inductive-step}

In this section we provide the (long) proof of Lemma~\ref{blowup_matching}. To simplify the presentation, we first introduce some notation.

\begin{defn} We colour an edge {\it green} if and only if it is not a loop and it was claimed by Avoider
in her second or later round. 
At any moment of the game we call a subgraph of Avoider's graph green if it is non-empty and all of its edges are green. 
Throughout the game we denote by $G$ the {\em green graph}, that is, the graph induced by the set of green edges. 
Note that $G$ has no isolated vertices and no loops. A connected component of $G$ is called a {\em green component}. 
{If $H$ contains a cycle,} 
we denote by $G_k$ the subgraph of $G$ induced by the set of all green edges with both endpoints in $\bigcup_{i=1}^k V_i$.
In other words, $G_k$ is induced by the set of green edges in the blow-up of the cycle $C_k$.
\end{defn}

\begin{proof}[\textbf{Proof of Lemma~\ref{blowup_matching}}]
	
For each edge $ij \in E(H)$ with $i\ne j$, let 
$$
m_{i,j}=\Big\lc \frac{|V_i||V_j|}{64(b+1)h}\Big\rc.
$$
Note that $m_{i,j} > 1$, as by \eqref{bigprod_match}, {in case of $H$ with a cycle we have}
$$
|V_i||V_j| \ge \frac{(16^h(b+1)h)^{h-1}}{(b/8h)^{h-2}}\ge 256(b+1)h,
$$
{while if $H$ is a tree, then for $t\in[h]\setm\{i,j\}$ we have
$$
|V_i||V_j| \ge \frac{(16^h(b+1)h)^{h-2}|V_t|}{|V_t|(b/8h)^{h-3}}\ge 256(b+1)h.
$$
Thus in both cases
\begin{equation}\label{bignn}
|V_i||V_j| \ge 256(b+1)h\quad\text{for every } i\neq j. 
\end{equation}
}

 Let $ij$ be an arbitrary edge in $H$ with $i\ne j$. Avoider starts the game. After her first move we denote by $\cM_{i,j}$ a maximal matching (possibly empty) formed by Avoider's edges in 
$E(V_i,V_j)$. Let $U_{i}$ and $U_j$ be the sets of all vertices saturated by $\cM_{i,j}$ in $V_i$ and $V_j$ respectively. 
If $|\cM_{i,j}| \ge m_{i,j}$, we define 
$V'_{i}=U_{i}$, $V'_{j}=U_j$, and $V'_r=V_r$ for $r \neq i,j$. 
These vertex sets and the family $\F=\cM_{i,j}$ of paths of length $1$
clearly possess the required properties (P\ref{vprime})--(P\ref{manypaths}). 
So from now on we can assume 
\begin{equation}\label{matching-size}
|\cM_{i,j}|<m_{i,j} \quad \text{for every $ij \in E(H)$ with $i\ne j$}.
\end{equation}

Fix $\{\ell, t\} \in E(H)$ such that $\ell \ne t$ and $m_{\ell,t}=\max_{ij \in E(H), i\ne j}m_{i,j}$. Since $\cM_{\ell,t}$ is a maximal matching, Avoider has no edges between $V_{\ell}\setm U_{\ell}$ and $V_t \setm U_t$.  
Furthermore, we learn from \eqref{matching-size} and \eqref{bigprod_match} that $|V_{\ell}\setm U_{\ell}| \ge |V_{\ell}|-m_{\ell,t} \ge 3|V_{\ell}|/4$. Similarly, $|V_t\setm U_t|>3|V_t|/4$.
Thus we can find two subsets $W_{\ell}\subs V_{\ell}\setm U_{\ell}$ and $W_t\subs V_t\setm U_t$ of sizes
$\lfloor |V_{\ell}|/2\rfloor$ and $\lfloor |V_t|/2\rfloor$, respectively. 
We call the edge-set $E(W_{\ell},W_t)$ {\em the dustbin}.

In what follows we first present Enforcer's strategy separately for graphs $H$ with a loop {or without a cycle}, 
with a cycle $C_2$, and with longer cycles. We then analyse the properties of Avoider's graph.\\

\noindent {\bf Case 1:}  {$H$ is a tree or $H$ has a loop.}

In his first move Enforcer selects $r$ arbitrary edges from the dustbin.
Starting with the second round, 
Enforcer responds to Avoider's move $e\in E(V_i,V_j)$ in the following way.  
If $i=j=1$ and there are 
$5h\cdot m_{\ell,t}$ loops in $V_1$ which have been claimed by Avoider, or if 
$i\ne j$ and there are $4m_{i,j}$ green edges in $E(V_i,V_j)$, then Enforcer 
{\it stops} the game. Otherwise he plays as follows.

\begin{enumerate}[{\rm (a)}]  
\item %\label{inc_xy}
Enforcer selects all free edges in $E(V_i,V_j)$ which are incident to $e$.

\item For every vertex $v$ in the green components other than the green component containing $e$, Enforcer claims all free edges between $v$ and $e$.

\item %\label{trash}
If Enforcer still has to choose more edges, he picks them from the dustbin $E(W_{\ell},W_t)$. 
\end{enumerate}
\noindent
Note that items (a) and (b) only apply if $e$ is not a loop. 
The following claim details the outcome of the game in {Case 1.} 
\begin{claim}\label{claimGbis-case1}
Enforcer can follow the strategy described above. Moreover, throughout the game, the green graph $G$ has the following properties.
\begin{enumerate}[{\rm (i)}]  
\item \label{deg_match_bis}
The edge-set $E(G)\cap E(V_i,V_j)$ is a matching for every edge $ij \in E(H)$ with $i\neq j$.
		
\item \label{separ_bis}
Before every Avoider's move there is no free edge between distinct green components.
		
\item \label{enf_comp_bis}
Enforcer has selected no loops.

\item
Every green component  is a canonical copy of a subgraph of $H$.

\item\label{enf_edge_bis}
Every edge selected by Enforcer is incident to a green edge or belongs to the dustbin.
\end{enumerate}
\end{claim}
\begin{proof}
It is straightforward from Enforcer's strategy that, during the game, the green graph $G$ satisfies conditions (i)--(v). 

We now show that he can follow the strategy. Indeed, he can easily follow rules (a) and (b) because the number of edges incident to $e$ is at most $2h\max_{r\in[h]}|V_r|<b$, due to \eqref{bigprod_match}. To prove that Enforcer can follow (c), it suffices to show that at every moment of the game until it stops, the dustbin contains at least $b$ free edges. Because the number of edges claimed by Avoider from the second round onward is at most 
$5h\cdot m_{\ell,t}+\sum_{ij\in E(H)} 4 m_{i,j}\le 9h\cdot m_{\ell,t}$ (due to Enforcer's stopping condition), the number of free edges in the dustbin is at least
\begin{align*}
|W_{\ell}||W_t|-r-9(b+1)h\cdot m_{\ell,t} &\ge \frac{15|V_{\ell}|}{32}\cdot\frac{15|V_t|}{32}-b-9(b+1)h\cdot \left(\frac{|V_{\ell}||V_t|}{64(b+1)h}+1\right)\\ &\overset{\eqref{bignn}}{\ge} b+\left(\frac{15^2}{32^2}-\frac{9}{64}-\frac{10}{256}\right)|V_{\ell}||V_t|>b,
\end{align*} 
where in the first inequality we estimate $|W_{\ell}|=\lfloor \frac{|V_{\ell}|}{2}\rfloor \ge \frac{15|V_{\ell}|}{32}$ for $|V_{\ell}| \ge 16$, and $|W_t| \ge \frac{15}{32}|V_t|$.    
\end{proof}

\noindent
We next present Enforcer's strategy in the case when $H$ has a $C_2$.\\

\noindent {\bf Case 2:} {$C_2\subs H$.} %$k=2$.

By symmetry, we only need to handle the cases $\{\ell,t\}\neq\{1,2\}$ and $(\ell,t)=(1,2)$.

\noindent {\bf Case 2.a:} $\{\ell,t\}\neq\{1,2\}$.

In his first move Enforcer selects $r$ arbitrary edges from the dustbin.
Starting with the second round, 
Enforcer responds to Avoider's move $e\in E(V_i,V_j)$ in the following way.  
If the number of green edges in $E(V_i,V_j)$ is at least $4m_{i,j}$, then Enforcer {\it stops} the game. Otherwise, he plays according to the following strategy.

\begin{enumerate}[{\rm (a)}]  
\item %\label{inc_xy}
Enforcer selects all free edges $f\in E(V_i,V_j)$ such that $|e\cap f|=1$.\footnote{In particular, Enforcer will not pick the friend of $e$ if $e\in E(V_1,V_2)$.}

\item %\label{sep_comp}
For every vertex $v$ in green components other than the green component containing $e$, 
Enforcer selects all free edges between $v$ and $e$. 

%\item %\label{threat_xy}
%If $e\in E(V_1,V_2)$, then Enforcer selects every free edge in $E(V_1,V_2)$ which is incident to $e$, except the friend of $e$.  

\item %\label{trash}
If Enforcer still has some edges to choose, he selects them from the dustbin $E(W_{\ell},W_t)$. 
\end{enumerate}

\noindent {\bf Case 2.b:} $(\ell,t)=(1,2)$.

Since by assumption $H\ne C_2$ and $H$ is connected, we can assume that 
$\{2,3\}\in E(H)$. 
Due to \eqref{bigprod_match} and \eqref{matching-size}, there exist two subsets $W'_2\subs V_2$, $W'_3\subs V_3$ with 
$\lfloor |V_2|/8\rfloor$ and $\lfloor |V_3|/8\rfloor$ vertices, respectively, such that Avoider has no edges in $E(W'_2,W'_3)$. (The argument for the existence of $W'_2$ and $W'_3$ is the same as the one used for the sets 
$W_{\ell}$ and $W_t$). We call the edge-set $E(W'_2,W'_3)$ {\em the remainder-bin}.

Enforcer plays as follows. In the first round he begins by picking $\lf r/2\rf$ free couples of friends 
in the dustbin $E(W_{\ell},W_t)=E(W_1,W_2)$. 
If $r$ is odd, Enforcer finishes this round by picking an arbitrary edge in the remainder-bin $E(W'_2,W'_3)$. 
From the second round onward, Enforcer responds to Avoider's  move $e\in E(V_i,V_j)$ as follows.
If the number of green edges in $E(V_i,V_j)$ is at least $4m_{i,j}$, then Enforcer again {\it stops} the game. If this is not the case, he plays according to the following strategy.  

\begin{enumerate}[{\rm (a)}]  
\item %\label{inc_xy}
Enforcer selects all free edges $f \in E(V_i,V_j)$ with $|e\cap f|=1$.

\item %\label{sep_comp}
For every vertex $v$ in green components other than the green component containing $e$, 
Enforcer selects all free edges between $v$ and $e$. 

%\item %\label{threat_xy}
%Suppose that  $e\in E(V_1,V_2)$. 
%Then Enforcer selects every free edge in $E(V_1,V_2)$, incident to $e$, except for the friend of $e$.  

\item %\label{trash}
If Enforcer still has some, say $s$, edges to choose, he selects $\lf s/2\rf$ free couples of friends in 
$E(W_1,W_2)$. After this, he must take either zero edges or one edge, depending on the parity of $s$.

\item  
If Enforcer has to choose one more edge, he picks an arbitrary edge in $E(W'_2,W'_3)$. 
\end{enumerate}

\begin{claim}\label{claimGbis-case2}
Enforcer can follow the strategies described in Cases 2.a and 2.b. Moreover, during the game, the green graph $G$ has the following properties.
\begin{enumerate}[{\rm (i)}]  
\item \label{deg_match_bis}
The edge-set $E(G)\cap E(V_i,V_j)$ is a matching for every $\{i,j\}\neq \{1,2\}$, and
$E(G)\cap E(V_1,V_2)$ consists of pairwise-disjoint edges and copies of $C_2$.
		
\item \label{separ_bis}
Before every Avoider's move there is no free edge between distinct green components.
		
\item \label{enf_comp_bis}
If an edge in $E(V_1,V_2)$ is green, then its friend is not selected by Enforcer.

\item
Every green component  is a canonical copy of a subgraph of $H$.

\item\label{enf_edge_bis}
Every edge selected by Enforcer is incident to a green edge, or belongs to the dustbin or to the remainder-bin. 
\end{enumerate}
\end{claim}

\begin{proof}
It is easy to see that if Enforcer can carry out his game plan, then the green graph satisfies Properties (i)--(v). The proof that Enforcer can follow rules (a)--(c) in both cases is exactly as in Claim \ref{claimGbis-case1}, which we need not repeat here. 

To complete the proof, we show that at every step of the game until it stops, the remainder-bin $E(W'_2,W'_3)$ contains at least one free edge; 
this will clearly imply that Enforcer can perform (d) in Case 2.b. 
For $i=2,3$, let $W''_i$ be the set of all vertices of $W'_i$ not incident to green edges.
First observe that due to the stopping condition, the number of endpoints of green edges in $V_i$ is at most
$$
\sum_{j:\, ij\in E(H)} 4 m_{i,j}<
\sum_{j:\, ij\in E(H)} \frac{8|V_i||V_j|}{64(b+1)h}
\overset{\eqref{bigprod_match}}{<}
|V_i|/64.
$$
Thus $|W''_i|> \lf |V_i|/8\rf - |V_i|/64\ge 3|V_i|/64$.
Furthermore, the number of edges selected by Enforcer in $E(W''_2,W''_3)$ is not greater than
$1+\sum_{ij \in E(H)}4m_{i,j} \le 1+4h\cdot m_{1,2}$, since in every round Enforcer selects at most one edge in $E(W''_2,W''_3)$; indeed, he only selects an edge in $E(W''_2,W''_3)$ when he invokes rule (d) in Case 2.b. 
Therefore the number of free edges in the remainder-bin is not less than 
$$
|W''_2|\cdot |W''_3|-1-4h\cdot m_{1,2}>
\frac{9|V_2|\cdot|V_3|}{2^{12}}-\frac{|V_1|\cdot|V_2|}{4(b+1)}
\overset{\eqref{bigprod_match}}{>}
\frac{9|V_2|\cdot 4^h}{2^{12}}-\frac{|V_2|}{32} \geq 0.
$$
Hence the remainder-bin contains a free edge.
\end{proof}

\noindent
Finally we describe Enforcer's strategy in the case when the unique cycle of $H$ has length $k\ge 3$.
 
\noindent {\bf Case 3:} $C_k\subs H$ and $k\ge 3$.

In the first round Enforcer selects $r$ arbitrary edges from the dustbin.
In each subsequent round, Enforcer responds to Avoider's move $xy\in E(V_i,V_j)$ in the following way.
He {\it stops} the game if there are $4m_{i,j}$ green edges in $E(V_i,V_j)$ 
% for some $ij \in E(H)$, 
and otherwise he proceeds as follows.  

\begin{enumerate}[{\rm (a)}]  
\item\label{inc_xy}
Enforcer selects all free edges in $E(V_i,V_j)$ which are incident to $xy$.

\item\label{sep_comp}
For every vertex $v$ in green components other than the green component containing $xy$, 
Enforcer selects all the free edges among $xv$ and $yv$. 

\item\label{threat_xy}
Suppose that the green component of $G_k$ containing $xy$ 
is a canonical path $P$ on $k$ vertices. %"unique"  not very important (affects constant in b>8hn)
If 
$u \in V_s$ and $w \in V_{s+1\pmod k}$ are the endpoints of $P$ for some $s \in [k]$,
then Enforcer selects all free edges in $E(V_{s},V_{s+1\pmod k})\setminus \{uw\}$ which are incident to $uw$.

\item\label{path_trash}
Suppose that the green component of $G_k$ containing $xy$ %has a unique 
is a canonical path $P$ on $k-1$ vertices. 
Let $s \in [k]$ be such that $V_s \cap V(P)=\emptyset$, let
$u \in V_{s-1\pmod k}$ and $w \in V_{s+1\pmod k}$ 
be the endpoints of $P$,
and suppose that $\{u,w\}\cap (W_{\ell} \cup W_t) \ne \emptyset$.
Then Enforcer selects all free edges between $\{u,w\}$ and 
$V_s\cap(W_{\ell}\cup W_t)$.

\item\label{trash}
If Enforcer still has some edges to choose, he selects them from the dustbin $E(W_{\ell},W_t)$. 
\end{enumerate}

\noindent
The following claim {summarises} the effects of Enforcer's strategy. 

\begin{claim}\label{claimG}
Enforcer can follow the above strategy. Furthermore, the green graph $G$ has the following properties throughout the game.	
\begin{enumerate}[{\rm (i)}]  

\item\label{deg_match}
The edge-set $E(G)\cap E(V_i,V_j)$ is a matching for every $ij\in E(H)$.
		
\item\label{separ}
Before every Avoider's move there is no free edge between distinct green components.
		
\item\label{enf_edge}
Let $e = uw$ be Enforcer's edge and suppose that $e \in E(V_i,V_j)$. Then at least one of the following holds.
\begin{enumerate}
	\item There is a green edge in $E(V_i,V_j)$ which is incident to $e$. 
	\item $u,w \in V(G)$.
	\item $i,j \in [k]$ and there is a canonical green path on $k$ vertices in $G_k$ between $V_i$ and $V_j$ (in particular, $j-i \equiv \pm 1 \mod{k}$), one of whose endpoints is $u$ or $w$.
	\item $u$ is an endpoint of a green path on $k-1$ vertices in $G_k$ and $w \in W_{\ell} \cup W_t$, or vice versa. 
	\item $e \in E(W_{\ell},W_t)$.
\end{enumerate}
In particular, either $e \in E(W_{\ell},W_t)$ or $e$ is incident to a green edge. 
			
\item\label{nogiant}
Every component of $G_k$ has at most $k$ vertices.
		
\item\label{enf_comp}
No Enforcer's edge has both ends in the same component of $G_k$.
\end{enumerate}
\end{claim}

\begin{proof}
In much the same way as in the proof of Claim \ref{claimGbis-case1}, we learn that Enforcer can follow the strategy. It remains to show that the green graph satisfies properties (i)--(v).
 
Properties (\ref{deg_match}) and (\ref{separ}) follow immediately from rules (\ref{inc_xy}) and (\ref{sep_comp}), respectively.
These two properties and the fact that $H$ is unicyclic imply that at every moment:
\begin{itemize}
\item
the graph $G_k$ is a union of vertex-disjoint non-trivial paths and cycles;
\item
if a path in $G_k$ has $k$ vertices or less, then it is a canonical path;
\item
every green cycle has at least $k$ vertices, and if it has exactly $k$ vertices, then it is a canonical cycle;
\item
if there is no path on $k+1$ vertices in $G_k$, then every green component of $G$ is a canonical copy of a subgraph of $H$;
\item 
Avoider cannot decrease the number of green components nor the number of  $G_k$. 
\end{itemize}
We will use these observations implicitly many times in the proof. Properties \eqref{enf_edge} (a) -- \eqref{enf_edge} (e) follow immediately from rules (a) -- (e) in Enforcer's strategy.

We now prove property \eqref{nogiant} by induction on the number of rounds in the game. 
After the first round there are no green edges,
so the assertion is trivial. Suppose that the assertion is true
until the end of some round and consider Avoider's next move, say $xy\in E(V_i,V_{i+1})$, where 
$x \in V_i$ and $y \in V_{i+1}$, and indices are taken modulo $k$ (if Avoider claims an edge outside of $\bigcup_{i=1}^{k}{E(V_i,V_{i+1})}$ then this clearly does not affect property \eqref{nogiant}).     
To prove that property \eqref{nogiant} holds after Avoider's move $xy$, assume by contradiction that  
the component of $G_k$ containing $xy$ has more than $k$ vertices. 
By the induction hypothesis, \eqref{nogiant} held before Avoider's move $xy$. Moreover, since Avoider cannot join two 
 $G_k$, either $x$ or $y$, say $y$, was not incident to an edge of $G_k$ before the move $xy$. This implies that 
 the component of $G_k$ containing $xy$ has $k+1$ vertices (so it is a path), and that before Avoider's move $xy$, 
the component of $G_k$ containing $x$ was of size $k$. By property \eqref{deg_match}, this component had to be 
a canonical path $x_1,\dots,x_k$ in which $x$ is an endpoint, say $x = x_1 \in V_i$, and the other endpoint $x_k$ is in $V_{i+1}$. 
Consider the time just after the last edge $e$ of the path $x_1,\dots,x_k$ was claimed by Avoider. 
At that time, the path $x_1,\dots,x_k$ was the green component of $G_k$ containing $e$.
By rule \eqref{threat_xy} of Enforcer's strategy, he then claimed all edges
between $V_i$ and $V_{i+1}$ other than $x_1x_k$, including the edge $xy$. 
Hence $xy$ could not have been claimed by Avoider at a later time, a contradiction. 

To finish the proof of the claim, we now derive property \eqref{enf_comp} from properties
\eqref{deg_match}--\eqref{nogiant}. 
Assume to the contrary that the endpoints of some Enforcer's edge $uv$ are in the same component of $G_k$, 
and suppose that  $u \in V_i$ and $v \in V_{i+1}$, with indices taken modulo $k$.  
Since every component of $G_k$ is a canonical path or cycle, 
the component of $G_k$ containing $u,v$ is a canonical green path $u,x_2,\dots,x_{k-1},v$. 
We will consider several cases, depending on the rule according to which Enforcer selected $uv$. 

If $uv$ were selected according to rule \eqref{inc_xy}, then there must be an Avoider's edge in $E(V_i,V_j)$ 
which is incident to either $u$ or $v$. But then the green component of $G_k$ containing $u,v$ has at least $k+1$ vertices, which contradicts \eqref{nogiant}. 

If $uv$ were selected according to rule \eqref{sep_comp}, then at some point in the game, $u$ and $v$ were in different green components. However, as $u$ and $v$ are in the same green component at the end, Avoider must have joined these two components, which is impossible. 

If $uv$ were selected according to rule \eqref{threat_xy}, then at the moment just before it was claimed, in $G_k$ there was a  $k$-vertex path between $V_i$ and $V_{i+1}$ with one endpoint in $\{u,v\}$ and the other not in $\{u,v\}$. This implies that the component of $G_k$ containing $u,v$ has at least $k+1$ vertices, in contradiction to \eqref{nogiant}.  

Suppose that $uv$ were selected according to rule \eqref{trash}, 
and assume without loss of generality that $u \in W_{\ell}$ and $v \in W_t$.
Consider the edge of the path $u,x_2\dots,x_{k-1},v$ that Avoider claimed the latest. 
Since Avoider cannot join  $G_k$, this edge must be either $u x_2$ or $x_{k-1}v$; 
assume without loss of generality that it is $u x_2$. 
Consider the situation immediately after Avoider claimed the last edge $e$ of the path $P=x_2,x_3,\dots,x_{k-2},x_{k-1},v$. 
At that time, $P$ was the component of $G_k$ containing $e$, $V(P) \cap V_i = \emptyset$ and $v \in W_t$. 
According to rule \eqref{path_trash} of his strategy, Enforcer immediately claimed all edges in $E(\{x_2,v\},W_{\ell})$, including $ux_2$. Hence Avoider could not have claimed $ux_2$ at a later time, a contradiction. 

Assume that $uv$ were selected according to rule \eqref{path_trash}, and consider the situation immediately before Enforcer's turn in which he claimed $uv$. Without loss of generality, we may assume that 
$v \in W_{\ell} \cup W_t$, and that at this time in the game, the component of $G_k$ containing $u$ is a canonical $(k-1)$-vertex path $P$ satisfying
$V_{i+1} \cap V(P) = \emptyset$. In particular, the edge $x_{k-1}v$ must be free at this time. By property \eqref{deg_match}, we must have 
$P = u,x_2,\dots,x_{k-1}$. By rule \eqref{path_trash}, Enforcer must claim the edge $x_{k-1}v$ at the same turn he claims $uv$, contradicting the fact that Avoider claimed $x_{k-1}v$.  
\end{proof}

{\noindent We are now in a position to finish the proof.}\newline
\textbf{Deriving the lemma from Claims \ref{claimGbis-case1}--\ref{claimG}}. Let $G$ denote the green graph at the end of the game. According to the stopping condition, we know that $|E(G)\cap E(V_i,V_j)| \le 4m_{i,j}$ for all $ij \in E(H)$, and either Avoider's graph contains at least $5h\cdot m_{\ell,t}$ loops or $|E(G)\cap E(V_i,V_j)|=4m_{i,j}$ for some pair $ij \in E(H)$ with $i\ne j$.
So the proof falls naturally into two cases.

Let us first consider the case when Avoider's graph has at least $5h\cdot m_{\ell,t}$ loops. In particular, we must have $k=1$. Let $\F$ be the set of all loops in Avoider's graph whose endpoints are not incident to a green edge. 
Define $V'_1=V(\F)$, and 
$V'_i=V_i\setm(V(G)\cup W_{\ell}\cup W_t)$ for every $i>1$.
Due to the construction, property (P\ref{vprime}) is satisfied. 

We next verify (P\ref{bigsets}). Fix $i\in \{2,\ldots,h\}$. The number of green edges incident to vertices of $V_i$ is at most
$$
\sum_{j:\, ij\in E(H)} 4 m_{i,j}<
\sum_{j:\, ij\in E(H)} \frac{8|V_i||V_j|}{64(b+1)h}
<|V_i|/8,
$$
where we used \eqref{bigprod_match}.
Furthermore, $|V_i\cap (W_{\ell}\cup W_t)|=\max\{|V_i\cap W_{\ell}|,|V_i\cap W_t|\} \le |V_i|/2$. Hence, we have
$
|V'_i|=|V_i\setm (V(G)\cup W_{\ell}\cup W_t)|> |V_i|-|V_i|/2-|V_i|/8>|V_i|/4,
$
as required. 

Properties (P\ref{no_enf}) and (P\ref{no_enf_threat}) are direct consequences of Claim \ref{claimGbis-case1} (v) and Claim \ref{claimGbis-case1} (iii), respectively. Finally, property (P\ref{manypaths}) holds since
$$
|\F| \ge 5h\cdot m_{\ell,t}-\sum_{1j\in E(H), j>1}4m_{1,j} \ge h\cdot m_{\ell,t} \ge \max_{ij \in E(H), i \ne j}\frac{|V_i||V_j|}{64(b+1)h}.
$$

For the rest of the proof, we consider the case when 
$|E(G)\cap E(V_i,V_j)|=4m_{i,j}$ for some pair $ij \in E(H)$ with $i\ne j$. 
Fix such a pair $i,j$.
From Claims \ref{claimGbis-case1} (i), \ref{claimGbis-case2} (i) and \ref{claimG} (i), it follows that $E(G)\cap E(V_i,V_j)$ contains a matching $\cM$ of size $2m_{i,j}$. Let $s$ be the greatest number such that there exists a tree 
$T_s \subseteq H$ on $s$ vertices, 
and a family $\F_0$ of at least 
$$\frac{2\prod_{i\in V(T_s)} |V_i|}{(64(b+1)h)^{s-1}}$$ 
vertex-disjoint canonical green copies of $T_s$ in 
${\mathbb B}_{H}(V_1,\ldots,V_h)$. Clearly 
$s\ge 2$ because $\cM$ is a family of trees on two vertices with the required property. 
Let $\F$ be the family of all trees in $\F_0$ which are maximal in $G$, where a tree in $\mathcal{F}_0$ is maximal if it is not contained in a canonical copy of some tree $T_s \subsetneq T \subseteq H$. 
For every $i\in V(T_s)$, set $V'_i=V_i \cap \bigcup_{F \in \F}{V(F)}$. For $i\notin V(T_s)$, define

\begin{equation*}
V'_i=
\begin{cases}
V_i\setm(V(G)\cup W_1\cup W_2\cup W'_2\cup W'_3) &\quad \text{if $k=2$ and $\{\ell,t\}=\{1,2\}$},\\
V_i\setm(V(G)\cup W_{\ell}\cup W_t) &\quad \text{otherwise}. 
\end{cases}
\end{equation*}

Property (P\ref{vprime}) is obvious from the definitions.
By a similar argument as in the previous case, we can show that property (P\ref{bigsets}) holds. 
The only difference in calculations here is the case that $k=2$ and $\{\ell,t\}=\{1,2\}$. In this case, if $2\notin V(T_s)$ then
 $$
|V'_2|\ge |V_2|-|V_2\cap V(G)|-|W_2|-|W'_2|> |V_2|-|V_2|/8-|V_2|/2-|V_2|/8={|V_2|/4}.
$$

Property (P\ref{no_enf_threat}) follows immediately from Claims~\ref{claimGbis-case1} (\ref{enf_comp_bis}), \ref{claimGbis-case2} (iii) and \ref{claimG} (\ref{enf_comp}), 
 and the fact that the trees in $\F$ are canonical.

Next we verify property (P\ref{no_enf}). For simplicity of exposition, we assume that $k\ge 3$, as the proofs for the cases $k=1$ and $k=2$ are similar (and much easier). Suppose to the contrary that Enforcer has occupied an edge 
$uw\in E(V'_i,V'_j)$ (where $u \in V'_i$ and $w \in V'_j$) with $ij \in E(H)$ and $j\in V(H)\setm V(T_s)$. 
Suppose first that $i\notin V(T_s)$. Then $u,w \notin V(G)\cup W_{\ell}\cup W_t$ by the definition of $V'_i$ and $V'_j$.
In view of Claim~\ref{claimG} (\ref{enf_edge}), the edge $uw$ was not selected by Enforcer, a contradiction. Suppose, then, that $i\in V(T_s)$. In this case, $u\in V'_i$ is a vertex of a green tree in $\F$, and $w \notin W_{\ell} \cup W_t \cup V(G)$.
Since trees in $\F$ are maximal in $G$, there are no green edges in $E(V_i,V_j)$ that are incident to $u$. 
As $w \notin W_{\ell} \cup W_t \cup V(G)$, it follows that $uw\notin E(W_{\ell},W_t)$ and $uw$ is not incident to a green edge in 
$E(V_i,V_j)$. 
Claim~\ref{claimG} (\ref{enf_edge}) thus implies that 
$i,j \in [k]$ and $u$ is an endpoint of a canonical $k$-vertex path $P$ in $G_k$ between $V_i$ and $V_j$. Since $i \in V(T_s)$ and $j \in V(H) \setminus V(T_s)$, the tree $F \in \mathcal{F}$ containing $u$ can be extended to a canonical copy of some tree $T_s \subsetneq T \subseteq H$ by adding vertices of the path $P$. This contradicts the maximality of $F$.    

It remains to verify property (P\ref{manypaths}). If $V(T_s)=V(H)$, then  
$|\F_0|=|\F|>\frac{\prod_{i\in V(T_s)} |V_i|}{(64(b+1)h)^{s-1}}$ as required. Now suppose that $V(T_s)\neq V(H)$. Fix a tree $T$ such that $T_s\subset T\subset H$ and $V(T)\setminus V(T_s)=\{j\}$.
By Claims \ref{claimGbis-case1} (i), \ref{claimGbis-case2} (i) and \ref{claimG} (i), every two canonical green copies of $T$ in ${\mathbb B}_{H}(V_1,\ldots,V_h)$ are vertex-disjoint.
By the maximality of $s$, 
the number of canonical green copies of $T$ in ${\mathbb B}_{H}(V_1,\ldots,V_h)$ is less than
$$\frac{2\prod_{i\in V(T)} |V_i|}{(64(b+1)h)^{s}}\le \frac{|V_j||\F_0|}{64(b+1)h}\overset{\eqref{bigprod_match}}{<}\frac{|\F_0|}{64h}.$$
Thus the number of trees in $\F_0$ which are not maximal in $G$ is at most $e(H)\cdot \frac{|\F_0|}{64h}=\frac{|\F_0|}{64}$, giving 
$$
|\F|>\tfrac12 |\F_0|\ge \frac{\prod_{i\in V(T_s)} |V_i|}{(64(b+1)h)^{s-1}}.
$$
{The proof of Lemma \ref{blowup_matching} is at long last complete.} 
\end{proof}

\section{Concluding remarks and open problems}
\label{sec:conc_remarks}

We {have} proved that $\fl_H(n)=\Theta\big(n^{1/m(H)}\big)$ for every graph $H$ with {at least two edges and $m(H)\le 1$}. 
We believe that this is true for every (not necessarily connected) graph with at least two edges. 

\begin{conj}\label{conj_lower}
For every graph $H$ with at least two edges, {one has} 
$$\fl_H(n)=\Theta\big(n^{\frac{1}{m(H)}}\big).$$
\end{conj}
In order to prove this conjecture, it is enough to find a good strategy for Enforcer, since Avoider's part 
follows from Theorem~\ref{lowerH}. 

As for the upper threshold of the $H$-game, the problem seems more complicated. 
We showed that for unicyclic graphs $H$, the upper threshold $\fu_H(n)$ is of order 
$n^{v(H)/(v(H)-1)}$ for infinitely many values of $n$. We conjecture that this is the correct order for all values of $n$. 
This problem is open even if $H$ is {a triangle and seems to require a major strengthening of the number theoretic tools}. 

\begin{conj}\label{conj:cycle}
For every unicyclic graph $H$, {one has}  
$$\fu_{H}(n)=\Theta\big(n^{\frac{v(H)}{v(H)-1}}\big).$$
\end{conj}

We believe that for general graphs $H$, the order of the upper threshold \emph{cannot} always be expressed by a simple formula 
which depends only on the number of vertices and edges of $H$ and, possibly, the parameters $m(H)$ and $m'(H)$. 
Let $H-e$ denote the graph obtained %by 
{from} $H$ by deleting its edge $e$ (together with a vertex of degree one, if 
this vertex is an end of $e$). We pose the following conjecture. 

\begin{conj}\label{conj:H-minus}
For every graph $H$ with at least three edges, {one has} 
$$\fu_H(n)=\Theta\big(\max\limits_{e\in H} \fl_{H-e}(n)\big).$$
\end{conj}
\noindent
In view of \eqref{thr_stars}, {Conjecture~\ref{conj:H-minus} } is true for stars. {From \eqref{eqn:mprime}, Theorems~{\ref{lowerH}}, \ref{lowerm1} and \ref{upperCycle}(iii)}, we see that it also 
holds for infinitely many values of $n$ in the case of 
unicyclic graphs.  

Finally, let us comment on the monotone version of Avoider--Enforcer games, introduced in \cite{ae10}.
Recall that in a monotone $(1:b)$ Avoider--Enforcer game, the players select \emph{at least} 1 
and  \emph{at least} $b$ board elements per turn, respectively, and the threshold of the game is the greatest $b$ such that 
Enforcer has a winning strategy. Our strategy for Enforcer, presented in Section~\ref{sec:blowup}, can be applied
to monotone $H$-games as well, though the analysis is slightly more complicated.  
This way, in view of a general upper bound on the threshold $f_H^{\text{mon}}(n)$ in monotone Avoider--Enforcer  $H$-games proved in
\cite{af14}, one can obtain that if $H$ is a unicyclic graph, then $f_H^{\text{mon}}(n)=\Theta\big(n^{\frac{v(H)}{v(H)-1}}\big)$.

\section*{Acknowledgements}

Tuan Tran is supported by the Humboldt Research Foundation, and by the GACR grant GJ16-07822Y. This work was partially done while he was affiliated with the Institute of Computer Science of the Czech Academy of Sciences, with institutional support RVO:67985807. Lior Gishboliner is supported by ERC Starting Grant 633509. The authors wish to thank the organisers of the {\em TAU-FUB Workshop on Positional Games}, hosted 
by the Freie Universit\"at Berlin in 2016, where this work was initiated. Finally, the authors wish to thank the anonymous referees for their valuable suggestions which improved the presentation of this paper.
%%%%%  bibliography

\bibliographystyle{amsplain}

\appendix
\section{Number theoretic tools: Proofs}
\label{app:number_theory_tools}

In this appendix we provide the proofs of Lemmata \ref{bigrem}, \ref{div} and \ref{div2}.

\begin{proof}[\textbf{Proof of Lemma \ref{div}}.]
	Let $r_1$ and $r_2$ be two natural numbers such that
	$r_1/r_2=\alpha$ if $\alpha\le 1$, and $r_1/r_2=\alpha-1$ if $1<\alpha \le 2$.
	Suppose that $C\in\N$ is some large constant. 
	For every odd integer $k\ge 1$, we put $n=2(Ck)^{r_2}+2$, and 
	define $q=k^{r_1}$ if $\alpha \le 1$ and $q=(n+1)k^{r_1}$ if $1<\alpha\le 2$. By a simple calculation one can verify that $n^{\alpha}/(4C^{r_1})\le q\le
	2n^{\alpha}/C^{r_1}$.
	Moreover, $q$ is odd, and $q\mid\binom n2-1$ since $\binom n2-1=(n+1)(n-2)/2\,$.
\end{proof}
\begin{proof}[\textbf{Proof of Lemma \ref{div2}}.]
	Define $x=
	\lfloor \frac12 cn^{\alpha-1}\rfloor$ if $\lfloor
	\frac12 cn^{\alpha-1}\rfloor$ is odd, and $x=\lfloor \frac12 cn^{\alpha-1}\rfloor-1$ if $\lfloor \frac12 cn^{\alpha-1}\rfloor$ is even. 
	Pick $k\in\{1,2,\ldots,2x\}$ so that $n-k-1$ is divisible by $2x$.
	In particular, $n+k$ is odd.
	Let $t=\binom{k+1}{2}$ and $q=(n+k)x$.
	As $n+k$ and $x$ are odd, so is $q$. 
	Moreover, $q\le (n+cn^{\alpha-1})cn^{\alpha-1}/2<cn^{\alpha}$, and
	$q\ge n\cdot (\frac12 cn^{\alpha-1}-2)>\frac13
	cn^{\alpha}$ for sufficiently large $n$.
	Finally, we have $t\le (cn^{\alpha-1}+1)cn^{\alpha-1}/2<c^2 n^{2\alpha-2}$ for $n$ sufficiently large,
	and
	$q\mid\binom n2-t$ since $\binom n2-t=(n+k)(n-k-1)/2\,$ and %$x\mid n-k-1$
	{$2x\mid n-k-1$}.  
\end{proof}

The rest of this section is devoted to the proof of Lemma \ref{bigrem}.
Given integers $s\ge 1$ and $m$, we define $r_s(m)$ as the unique integer $j$ such that $-s/2<j\le s/2$ 
and $m\equiv j\text{ mod } s$.
It is easy to see that 
\begin{equation}\label{addyt}
|r_s(m+m')|\le |r_s(m)|+|r_s(m')|\quad\textrm{and}\quad |r_s(m)|\le |r_s(m+m')|+|r_s(m')|. 
\end{equation}

In the proof of Lemma \ref{bigrem}, we shall use the following simple observation.
\begin{obs}\label{smallrs}
	Let $s,j,y$ and $m$ be natural numbers such that $j\le s/2$, $y<s/4$, $j<ys/(4m)$ and $|r_{s-j}(m)|\le y$. Then $|r_s(m)|\le 2y$.  
\end{obs}

\begin{proof}
	If $m\le s/4$, then $r_{s-j}(m)=m=r_s(m)$ and the assertion follows.
	Now suppose that $m>s/4$. Clearly there exists an integer $t$ such that $m=(s-j)t +r_{s-j}(m)$.
	As $|r_{s-j}(m)|\le y$, $y<s/4<m$, and $j<ys/(4m)$, one has  
	$jt=j\cdot\frac{m-r_{s-j}(m)}{s-j}
	<\frac{ys}{4m}\cdot\frac{m+y}{s-ys/(4m)}
	=y\cdot\frac{m+y}{4m-y}
	<y$.
	Combined with \eqref{addyt} we get 
	$$|r_{s-j}(m)-jt|\le |r_{s-j}(m)|+jt<2y<s/2,$$
	as $r_{s-j}(m)|\le y<s/4$. Since $m=st+r_{s-j}(m)-jt$, it follows that $|r_{s}(m)|=|r_{s-j}(m)-jt|<2y$.
\end{proof}

\begin{proof}[\textbf{Proof of Lemma \ref{bigrem}}.] 
As $\alpha>1$, we can write $\alpha=t+\beta$, in which $t\in \mathbb{N}$ and $0<\beta \le 1$. 
Let $C=C(t)$ be a sufficiently large integer, and let 
$$b=Cq, \quad x=\left\lf C^{1+\frac{\beta}{2t}} b^{\frac{1-\beta}{t}}\right\rf.$$ 
By the assumption, we obtain 
\begin{equation*}
c_2\cdot\left(\frac{b}{C}\right)^{t+\beta}\le N\le c_1\cdot\left(\frac{b}{C}\right)^{t+\beta}.
\end{equation*}
Let $(a_t,a_{t-1},\ldots,a_0)$ be the greatest, in sense of lexicographical order,  sequence of non-negative integers such that
\begin{equation}\label{nform}
N=a_0+a_1b+a_2b(b-x)+\ldots+a_tb(b-x)\cdots(b-(t-1)x).
\end{equation}
It is not difficult to see that $a_0,a_1,\ldots,a_{t-1} <b$, and 
$$
a_t=(1+o(1))\frac{N}{b^t}\le (1+o(1)) \frac{c_1b^\beta}{C^{t+\beta}}<\frac{b}{x}
$$ 
for sufficiently large constant $C$.

We will show that the required number $k$ can be found among 
$t+1$ numbers $b,b-x,\ldots,b-tx$. Suppose to the contrary that for every $k\in\{b,b-x,\ldots,b-tx\}$ the remainder of the division of $N$ by $k$ is at most $q$. As $q=b/C<(b-tx)/2\le k/2$, this implies $|r_k(N)|\le q$ for every such $k$. Combined with \eqref{nform}, we find
\begin{equation*}
q\ge |r_{b-ix}(N)|\overset{\eqref{nform}}{=} |r_{b-ix}(a_0+ixa_1+i(i-1)x^2a_2+\ldots + i!x^{i}a_i)|=r_{b-ix}(A_i)
\end{equation*}
for every $i\in \{0,1,\ldots,t\}$, where
\begin{equation*}
A_i:=a_0+ixa_1+i(i-1)x^2a_2+\ldots + i!x^{i}a_i.
\end{equation*}

For each $i\in \{0,1,\ldots,t\}$, we have
$$
ixA_i \le txA_t \le (t+2)!x^t(b+xa_t)=O(b^{2-\beta}),
$$
as $x=O(b^{(1-\beta)/t})$ and $a_t<b/x$. It follows that  $ixA_i<qb/4$. Since $q<b/4$, $ix=o(b)$, $ixA_i<qb/4$ and $|r_{b-ix}(A_i)| \le q$, we may apply Observation \ref{smallrs} to $s=b$, $j=ix$, $y=q$ and $m=A_i$ to conclude that
\begin{equation*}
|r_b(A_i)| \le 2q \quad \text{for each $i \in \{0,1,\ldots,t\}$}.
\end{equation*}

Since $A_i=a_0+ixa_1+i(i-1)x^2a_2+\ldots+i!x^i a_{i}=i!x^i a_{i}+\sum_{j=0}^{i-1} \binom ij j! x^j a_j$,
we get
$$|r_b(i!x^i a_{i})| \overset{\eqref{addyt}}{\le} |r_b(A_i)|+\sum_{j=0}^{i-1} \binom ij |r_b(j! x^j a_j)|.$$
As $|r_b(A_i)| \le q$, it follows that there exists a constant $d_i$ depending only on $i$ such that 
\begin{equation}\label{smalla}
|r_b(i!x^i a_{i})|<d_iq.
\end{equation}

However, in view of (\ref{nform}) and the definition of $x$, we have
$$t!x^t a_t=(1+o(1))t!C^{t+\beta/2}b^{1-\beta}\cdot \frac{N}{b^t}
\ge(1+o(1))t!C^{t+\beta/2}\frac{c_2b}{C^{t+\beta}}
=(1+o(1))t!C^{1-\alpha/2}c_2q>d_tq,$$
for $C$ sufficiently large.
On the other hand,
$$t!x^t a_t=(1+o(1))t!C^{t+\beta/2}b^{1-\beta}\cdot \frac{N}{b^t}
\le(1+o(1))t!C^{t+\beta/2}\frac{c_1b}{C^{t+\beta}}
=(1+o(1))t!\frac{c_1b}{C^{\beta}}<\frac{b}{2},$$
provided that $C$ is large enough.
We infer that  $r_b(t!x^t a_t)=t!x^t a_t>d_tq$, which contradicts (\ref{smalla}).
\end{proof}

\end{document}